\documentclass[11pt]{amsart}
\usepackage{hyperref}
\usepackage{color}
\usepackage{relsize}
\usepackage{etoolbox}
\usepackage{graphicx}
\usepackage{epsfig}
\usepackage{enumitem}
\usepackage[T1]{fontenc}
\usepackage[utf8]{inputenc}
\usepackage{enumitem}
\usepackage[a4paper,centering,top=25mm,bottom=25mm,bindingoffset=0mm]{geometry}
\numberwithin{equation}{section}       

\theoremstyle{plain}
\newtheorem{theorem}{Theorem}[section]
\newtheorem*{theorem*}{Theorem}
\newtheorem{prop}{Proposition}[section]

\newtheorem{remark}[prop]{Remark}
\newtheorem{coro}[prop]{Corollary}

\newtheorem{lemma}[prop]{Lemma}

\newtheorem*{mainthm}{Main Theorem}
\newtheorem*{mainthm-hyp}{Main Theorem (Hyperbolic)}
\newtheorem*{mainthm-par}{Main Theorem (Parabolic)}

\newtheorem{example}[prop]{Example}

\newenvironment{proofof}[1]{\medskip \noindent{\it Proof of #1.}}{ \hfill\qed\\ }
\usepackage{hyperref}

\makeatletter
\newcommand{\newreptheorem}[2]{\newtheorem*{rep@#1}{\rep@title}\newenvironment{rep#1}[1]{\def\rep@title{#2 \ref*{##1}}\begin{rep@#1}}{\end{rep@#1}}}
\makeatother

\theoremstyle{definition}
\newtheorem{definition}[prop]{Definition}

\newtheoremstyle{citing}
  {3pt}
  {3pt}
  {\itshape}
  {}
  {\bfseries}
  {.}
  {.5em}
  {\thmnote{#3}}

\theoremstyle{citing}

\theoremstyle{empty} 

\DeclareMathAlphabet{\mathpzc}{OT1}{pzc}{m}{it} 

\newcommand{\C}{\mathbb{C}}
\newcommand{\D}{\mathbb{D}}

\newcommand{\R}{\mathbb{R}}

\newcommand{\Z}{\mathbb{Z}}

\newcommand{\m}{p}

\newcommand{\cC}{\mathcal{C}}

\newcommand{\teta}{\widetilde{\teta}}

\newcommand{\eps}{\varepsilon}

\newcommand{\dist}{d}

\DeclareMathOperator{\diam}{diam}

\def\bv{\Big\vert}

\makeatletter
\patchcmd{\@addmarginpar}{\ifodd\c@page}{\ifodd\c@page\@tempcnta\m@ne}{}{}
\makeatother
\reversemarginpar

\begin{document}
\title[Transversality in hyperbolic and parabolic maps]{Transversality in the setting of hyperbolic and parabolic maps } %
\author{Genadi Levin, Weixiao Shen and Sebastian van Strien}

\maketitle

{\centering\footnotesize Dedicated to Lawrence Zalcman.\par}

\bigskip

\begin{abstract} In this paper we consider families of holomorphic maps defined on subsets of the complex plane,
and  show that the technique developed in
\cite{LSvS1}
to treat unfolding of critical relations can also be used to deal with cases where the critical orbit converges to a
hyperbolic attracting or a parabolic periodic orbit.   As before this result applies to rather general
families of maps,  such as polynomial-like mappings,  provided some lifting property holds.
Our Main Theorem states that either the multiplier of a hyperbolic attracting periodic orbit depends univalently
on the parameter and bifurcations at parabolic periodic points are generic, or one has persistency of periodic orbits
with a fixed multiplier.
\end{abstract}


%

\section{Introduction}
When studying families of maps defined on an open subset of the
complex plane, it is useful to have certain transversality properties. For example, do  multipliers of attracting periodic points
depend univalently on the parameter and  do  parabolic periodic  points undergo
generic bifurcations?   Building on a method developed in \cite{LSvS1}
 we establish  such transversality results in a very general setting.
 The conclusion of our  Main Theorem states that one has either {\em such transversality} or
 {\em persistency of periodic points with the same multiplier holds.}

 The key assumption in our Main Theorem is a so-called {\em  lifting property}
 defined in \S\ref{subsec:lifting}.  It turns
 out that this assumption is applicable in rather general settings, 
 including
 families of maps with an infinite number   of singular values,   such as polynomial-like mappings and also 
 maps with essential singularities.

Although the Main Theorem applies to complex maps, let us first mention applications
to certain families of real maps.
For example, consider the periodic doubling cascade associated to the family
$f_\lambda=\lambda x (1-x)$, $x\in [0,1]$ and $\lambda\in [0,4]$. It is well-known that the multiplier $\kappa(\lambda)$
of attracting periodic
orbit decreases in $\lambda$ diffeomorphically in each interval for which $\kappa(\lambda)\in [-1,1)$ and that one has
generic bifurcations when $\kappa(\lambda)=\pm 1$.
An application of our result is that
the same conclusion holds for families of the form  $f_\lambda(x)=\lambda f(x)$ and similarly for $g_c(z)=g(z)+c$
where $f$ and $g$ are rather general interval maps.
An important feature of our method is that we obtain
information about the sign of the derivative of $\kappa'$, namely that  $\kappa'>0$, see \S\ref{sec:realmaps}.

The main aim of this paper is to obtain results which apply to maps
which are defined only locally, e.g.,  having
essential singularities. It turns out that many more classical results can be recovered
also.
Before stating our results more formally, let us discuss previous results
and the approach that is used in this paper.


%


The study of the dependence of a multiplier on parameters goes back to
Douady-Hubbard  \cite{DH0,DH1}, who obtained the celebrated result (using Sullivan's quasiconformal surgery)  that the multiplier map $\kappa$ uniformizes hyperbolic
components of the Mandelbrot set.
Milnor \cite{Milnor1}  generalized this result to spaces of hyperbolic polynomials and Rees  \cite{Rees} to the space of degree two rational maps.  Douady-Hubbard theorem is equivalent to the local claim that $\kappa'\neq 0$ whenever $|\kappa|<1$.
In \cite{Le3} and independently in \cite{Ep2} the latter local result is generalized to spaces of polynomials and rational maps.

In \cite{Le3}, a polynomial or a rational function $f$ is considered which has $r$ cycles $O_i$ with corresponding multipliers $\kappa_i\in\overline\D=\{|\kappa|\le 1\}$
such that if, for some $i$, $\kappa_i=1$ then $O_i$ is not degenerate
and if $\kappa_i=0$ then $O_i$ contains a single critical point which is simple.
Now consider a local space $\Lambda$,  containing $f$,
of functions with a constant number of different critical values and with constant multiplicities at the critical points. Then the multipliers of those $r$ cycles contribute $r$ independent parameters
in the coordinate system of $\Lambda$, see \cite[Theorems 2 and 6]{Le3} for precise statements.
The proof, which is a development of \cite{Le0}, \cite{LSY}, \cite{Le1}, \cite{Le2},
see also \cite{Mak},
relies on a non-trivial identity \cite[Theorems 1 and 5]{Le3} between, on the one hand, a specific function associated to a given cycle $O$ of $f$, and on the other hand, the gradient $\nabla\kappa$ of the multiplier of this cycle as a function of the coordinates of $\Lambda$. If $\kappa=1$, a modification is needed by replacing the gradient
by 
$$\lim(1-\kappa(g))\nabla\kappa(g)$$ as $g\to f$
along $g$ having a cycle near $O$ with multiplier $\kappa(g)\neq 1$.
A calculation as in Lemma \ref{lem2.2}
shows that
the limit is equal to $$(f^q)''(a)\nabla g^q|_{g=f}(a)$$ (independently on $a\in O$).
In particular, if $f(z)=z^2+c_1$ so that
$\Lambda=\{g_w=z^2+w\}$,
it gives $$\frac{d}{dw}g_w^q(a_k)|_{w=c_1}\neq 0.$$ This recovers Douady-Hubbard's result \cite{DH1} that every primitive component of the Mandelbrot set has a simple cusp. The Main Theorem
 of the present paper, which can be found on page \pageref{thm:main},  is a far reaching generalization of this fundamental fact to suitable families of local maps.
The paper \cite{Le3} contains also a detailed historical account.

In turn, in Epstein's work \cite[Proposition 20]{Ep2} the above result is proven in the case when all $\kappa_i\in\overline\D\setminus\{0,1\}$ under the assumption that all critical points of $f$ are simple and $r$ is maximal (i.e., $r=\deg f-1$ for polynomials and $r=2\deg f-2$ for rational functions).
Epstein's work \cite{Ep2}, see also \cite{Ep3}, is a development after his earlier work \cite{Ep0}, \cite{Ep1}, is
 coordinate free and can also be used to prove transversality when there are critical relations. The techniques build on an approach pioneered by Thurston.
In Thurston's approach, when $f$ is a globally defined holomorphic map, $P$ is a finite $f$-forward invariant set containing
the postcritical set and the Thurston map
$\sigma_f\colon
\mbox{Teich}(\widehat{\C}\setminus P) \to \mbox{Teich}(\widehat{\C}\setminus P)$
 is defined by pulling back an almost holomorphic structure.
It turns out that $\sigma_f$ is contracting, see \cite[Corollary 3.4]{DH2}.
In  Thurston's result on the topological realisation of rational maps,  Douady \& Hubbard \cite{DH2}
use that the dual of the derivative of the Thurston map $\sigma_f$ is equal to the Thurston pushforward operator $f_*$ which acts on the space of quadratic differentials.
Epstein's approach is to exploit the  injectivity of  $f_*-id$.  One of the innovations in \cite{Ep2} is the
introduction of the space Teichm\"uller deformation space $\mbox{Def}^B_A(f)$ which
can be used in a wide range of settings.


\medskip
\def\MT{Milnor and \!Thurston}

There were other approaches for critically finite rational maps: Sullivan's pull-back argument \cite{Su,McS,MS}. 
Milnor and Thurston's \cite{MT}  
and Tsujii's \cite{Tsu0,Tsu} work. 
Their initial purpose was to prove the monotonicity of entropy for the real quadratic family.
All these and previously discussed methods break down if the map is not globally defined, see \cite{LSvS1a} for details.
In \cite{LSvS1,LSvS1a} we develop a method
which  goes back to  \MT's and
Tsujii's approaches in  allowing us to deal
with maps which are locally defined.
{\MT}   associated to the space of quadratic maps and
a given combinatorial  data on a periodic orbit $P\ni 0$ of $f_{c_1}$ of period $q$ where $f_{c}(z)\equiv z^2+c$,
a map
 which assigns to a $q$ tuple of points a new $q$ tuple of points, 
 $$F: (z_1,...z_q)\mapsto (\hat z_1,...,\hat z_q)$$
where $\hat z_q=0$ and
$f_{z_1}(\hat z_i)=z_{i+1 \!\!\! \mod q}$.
Since $F$ is many-valued, {\MT} considered a lift $\tilde{F}$ of this map to the universal cover and apply Teichm\"uller theory to show that $\tilde{F}$ is strictly contracting in the Teichmuller metric of the universal cover.

In \cite{LSvS1,LSvS1a} we bypass this 'global' approach
by rephrasing it locally.
This is done in the set-up of so-called  (locally defined) marked maps (and their local deformations)
and replacing $F$ by an (non-linear) operation which associates to a holomorphic motion $h_\lambda$ of a finite set $P$
another holomorphic motion $\hat h_\lambda$ of $P$ that we call its {\it lift}.
The {\it lifting property} is an assumption that sequences of successive lifts of holomorphic motions
are compact. In \cite{LSvS1}, \cite{LSvS1a} we considered the case that the set $P$ is finite and prove that assuming the lifting property
either some 'critical relation' persists along some non-trivial manifold in parameter space  or one has
transversality, i.e. the 'critical relation' unfolds transversally. In \cite{LSvS1}, \cite{LSvS1a} we also prove the lifting property for many interesting families of maps.
An important feature of our method is that for real maps we obtain {\em positive transversality}
originated in \cite{Tsu0,Tsu}:
some determinant is not only non-zero but has positive sign.

Here we  consider maps with a non-finite 'postcritical set' $P$, namely
locally defined maps with an attracting or a parabolic periodic orbit.
The results of this paper do not aim to replicate the results mentioned previously,
so we only consider the case of maps with a single attracting or parabolic periodic orbit.
The proof follows the same strategy as in \cite{LSvS1}, \cite{LSvS1a}.
Although we don't use the transfer operator which is defined in \cite{LSvS1}, \cite{LSvS1a}
as the linearisation of the lift operator (and is infinitely dimensional if $P$ is infinite),
the proof here has a strong 'flavor' of it, see the beginning of Sect \ref{sec:ThmA} and Appendix A.
As mentioned, the maps  we consider are allowed to  have essential singularities.

For other  results for transversality in the setting  of polynomial, rational or finite type maps
(which have at most a finite number of singular values), see 
\cite{Astor,BE, D, DH1, FG,Le00,Le4, LSvS2,Str}.

\subsection{Organisation of this paper}
In  \S\ref{sec:results} we will state the results of this paper, in  \S\ref{sec:outline} we will give an outline
of the three steps used in the proof, which will be given in \S\ref{sec:ThmA}-\ref{sec:ThmC}.
Applications to complex and real families will be then given in \S\ref{sec:complex}-\ref{sec:realmaps}.

\medskip 

{\bf Acknowledgments.} We thank Mitsu Shishikura and Michel Zinsmeister for very fruitful discussions. We also
thank Alex Eremenko for inspiring us to write this paper, by  posing the question which we answer in Theorem~\ref{thm:eremenko}. We also would like to thank  the anonymous referee. The authors acknowledge Adam Epstein for pointing out an error/typo at the very end of the proof of Lemma 4.1 in the Arxiv version of this paper.
This work was started when GL  visited Imperial College (London) and continued during his visit at the Institute of Mathematics of PAN (Warsaw). He thanks both Institutions for their generous hospitality. GL acknowledges the support of ISF grant No. 1226/17, WS acknowledges the support of NSFC grant No. 11731003, and SvS acknowledges the support of ERC AdG Grant No. 339523 RGDD.

\section{Statement of results}\label{sec:results}

\subsection{Marked maps and holomorphic deformations}
Let $U$ an open subset of $\C$ and $g\colon U\to \C$ be holomorphic on $U$. Assume that $c_1\in U$.
Then we say that $g$ is a {\em marked map w.r.t. $c_1$} if
$c_n=g^{n-1}(c_1)$ is well-defined for all $n\ge 1$ and
$\overline P\subset U$ where $P=\{c_n\}_{n=1}^\infty$.
%
%
%
%
We say that  $(g,G)_W$ is a  {\em local holomorphic deformation} of $g$ if:
\begin{enumerate}
\item $W$ is an open connected subset of $\C$ containing  $c_1$;
\item $G\colon (w,z)\in W\times U \mapsto G_w(z)\in  \C$ is a holomorphic map such that $G_{c_1}=g$;
\item $DG_w(z)\neq 0$ for all $z\in U$, $w\in W$.
\end{enumerate}
Here and later on $D=\frac{\partial}{\partial z}$, $D^k=\frac{\partial^k}{\partial z^k}$.
Note that  $c_1$ plays the role of a special point in the dynamical space $U$, but it is also a special point in the parameter space $W$.

\begin{example}
Let  $g \colon U\to \C$ be a marked map w.r.t. $c_1$.  Moreover, let $W=\C$ and
$G\colon W\times U \to \C$   be defined by $G(w,z)=G_w(z)=g(z)+(w-c_1)$. Then
$(g,G)_W$ is a local holomorphic deformation of $g$.
 \end{example}


\subsection{Holomorphic motions and the lifting property}\label{subsec:lifting}
Recall that $h_\lambda$ is a {\em holomorphic motion} of a set $X\subset \C$ over   $(\Lambda, 0)$
where $\Lambda$ is a domain in $\C$ which contains $0$, if
 $h_\lambda \colon X\to \C$ satisfies:
 (i)  $h_0(x)=x$, for all  $x\in X$,
 (ii)  $h_\lambda(x)\ne h_\lambda(y)$ whenever $x\ne y$ and $\lambda\in \Lambda$, $x,y\in X$ and
 (iii) $\Lambda \ni \lambda \mapsto h_\lambda(z)$ is holomorphic.
 Quite often we will consider the special case where $\Lambda$ is equal to $\D_r:=B(0,r)$
 where $B(a,r):=\{w\in\C: |w-a|<r\}$.

Let $K$ be a bounded  set such that $P\subset K\subset\overline{K}\subset U$  and $g(K)\subset K$.
The following notions of the {\it lift} and the {\it lifting property} were introduced and studied in \cite{LSvS1}, \cite{LSvS1a} (in the case of a finite set $K$):

\begin{definition}\label{def:liftingproperty}
Given a holomorphic motion $h_\lambda$ of $K$ over $(\Lambda, 0)$,
we say that a holomorphic motion $\hat{h}_\lambda$ of $K$ over $(\Lambda_0, 0)$ where $\Lambda_0\subset \Lambda$,
is a {\em lift}  of $h_{\lambda}$ for $(g, G)_W$ if for all $\lambda\in \Lambda_0$ and $x\in K$,
\begin{equation}\label{liftdef}
G_{h_\lambda(c_1)}(\hat{h}_{\lambda}(x))=h_{\lambda}(g(x)).
\end{equation}
We say that $(g,G)_W$ has the {\it lifting property for the set}
$K$ if the following holds: Given a holomorphic motion $h_\lambda^{(0)}$ of $K$ over $(\Lambda, 0)$, there exist $\eps>0$ and holomorphic motions
$h_\lambda^{(n)}$, $n=1,2,,\ldots$
of $K$ over $(\D_{\eps}, 0)$ such that for each $n\ge 0$,
\begin{enumerate}
\item $h_\lambda^{(n)}(c_1)\in W$ for each $\lambda \in \D_\eps$;
\item $h_\lambda^{(n+1)}$ is the lift of $h_\lambda^{(n)}$ over $(\D_\eps, 0)$
for $(g, G)_W$,
\item there exists $M>0$ such that $|h_\lambda^{(n)}(x)|\le M$ for all
$x\in K$, $\lambda\in \D_{\eps}$ and $n\ge 0$.
\end{enumerate}
\end{definition}
The next lemma and the subsequent remark clarify this notion:
\begin{lemma}  Given any holomorphic motion $h_\lambda$ of $K$ over $(\D_{\epsilon_0}, 0)$,
there is a sequence of holomorphic motions $h_\lambda^{(n)}$, $n\ge 0$, of $K$ where $h_\lambda^{(0)}=h_\lambda$ and $h_\lambda^{(n)}$ is the lift of $h_\lambda^{(n-1)}$ over $(\D_{\eps_n}, 0)$ with some maximal $\eps_n>0$.
\end{lemma}
\begin{proof}  It is enough to show that given a holomorphic motion $h_\lambda$ of $K$ over $(\D_{\epsilon_0}, 0)$, there exists $\epsilon_1>0$ (which depends on $\epsilon_0>0$ and $h_\lambda$) such that the lift $\hat{h}_\lambda$ of $h_\lambda$ for the set $K$ exists over $\D_{\epsilon_1}$.
    Indeed, as $\overline{K}$ is compact it follows from the definition of local holomorphic deformation that there are a finite open covering $\{B(x_j, r_j)\}$ of $K$ and $\epsilon_1>0$ such that, for $\lambda\in\D_{\epsilon_1}$ the following hold: $h_\lambda(c_1)\in W$, for each $j$ the map
    $G_{h_\lambda(c_1)}$ is injective on $B(x_j,2r_j)$ and $h_\lambda(g(y))\in G_{h_\lambda(c_1)}(B(x_j,2r_j))$ for $y\in B(x_j,r_j)$. Then define $\hat{h}_\lambda(y):=G_{h_\lambda(c_1)}^{-1}(h_\lambda(g(y))$ for $y\in B(x_j,r_j)$ where
    $G_{h_\lambda(c_1)}^{-1}$ is the inverse map to
     $G_{h_\lambda(c_1)}|_{B(x_j,2r_j)}$.
\end{proof}

\begin{remark}
The lifting property means that for any holomorphic motion $h_\lambda$ of $K$ over $(\Lambda, 0)$, there is $\epsilon>0$ such that in the previous lemma $\eps_n>\eps$ holds for all $n\ge 0$ and such that the family $\{h_\lambda^{(n)}(x)\}$ is uniformly founded on $\D_{\eps}\times K$.
\end{remark}
\begin{remark}
 Note that if $K_1\subset K_2$ are two forward invariant sets
then the lifting property for $K_2$ implies the lifting property for $K_1$ (this follows from Slodkowski's generalised lambda lemma \cite{Slod} and also \cite{AIM}).
\end{remark}

In~\cite{LSvS1,LSvS1a}, we studied the case when  the {\lq}postcritical set{\rq} $P$ is finite and showed that
if this lifting property holds then, under suitable circumstances,  critical relations unfold transversally.
The present paper considers the case where $P$ is an  infinite orbit
converging to an attracting or neutral periodic orbit and the main result shows that,  again,
the lifting property implies transversality.

\medskip

 \subsection{The statement of Main Theorem}
In this paper we will study marked maps $g$ such that $P=\{c_n\}_{n=1}^\infty$ is an {\em infinite} orbit of $g$ so that
$c_n$ converges to a periodic orbit
$\mathcal{O}=\{a_0, a_1, \ldots, a_{q-1}\}$ and consider a holomorphic deformation $(g,G)_W$ of $g$.
The main theorem of this paper shows that if $(g,G)_W$  satisfies  the lifting property, then
the orbit $\mathcal{O}$ depends {\lq}transversally{\rq} on the parameter $w$ (in a sense made precise below).

As usual, we say that $\mathcal{O}$ is {\em hyperbolic attracting} if $\kappa:=Dg^q(a_0)\in \D\setminus \{0\}$. We say that $\mathcal{O}$ is {\em non-degenerate parabolic} if there exists $l,\m \in \Z$, $\m\ge 1$, $(l,\m)=1$ such that $\kappa=e^{2\pi i l/\m}$ and
$D^{\m+1} g^{\m q} (a_0)\not=0$.  Let
\begin{equation}\label{eqn:Q}
Q(z):=\left.\frac{d}{dw} G_w^q(z)\right|_{w=c_1},
\end{equation}
which is a holomorphic function defined in a neighborhood of $\mathcal{O}\cup P$.

\begin{example}\label{ex:Q}
If $\kappa=1$ then non-degeneracy means that $D^2g^q(a_0)\ne 0$.
If, moreover,  $Q(a_0)\ne 0$ then
the Taylor expansion of $G_w^q(z)$ at $z=a_0, w=c_1$ is given by $G_w^q(z)=z+\alpha (z-a_0)^2 + \beta (w-c_1) + h.o.t.$ with $\alpha,\beta\ne 0$. Our  Main Theorem states that a certain lifting property implies that either $Q(a_0)\ne 0$ or that the parabolic periodic point persists for all parameters $w$ near $c_1$.
\end{example}

We will also use the following subset of the basin of $\mathcal{O}$:
%
\begin{equation}\label{eqn:Omegar}
\Omega_r:=\left\{z\in U: \sup_{n=0}^\infty \dist(f^n(z),\mathcal{O})\le r, \text{ and }
\lim_{n\to\infty} \dist(f^n(z), \mathcal{O})=0\right\} \mbox{ where } r>0.
\end{equation}
In the attracting case, $\bigcup B(a_j,\rho)\subset\Omega_r\subset\bigcup \overline{B(a_j,r)}$
whenever $r$ is small enough and some even smaller $\rho>0$. In the parabolic case, the Leau-Fatou Flower Theorem (see Lemma~\ref{lem:Omega} in Appendix B)   tells us that $(\bigcup B(a_j,\rho))\setminus \Omega_r$,  even though non-empty, is located in a very thin region near the repelling directions. 

\label{thm:main}
\begin{mainthm} Assume that $\mathcal{O}$ is either hyperbolic attracting or non-degenerate parabolic and that $(g,G)_W$ satisfies the lifting property for $P_{r_0}=P\cup \Omega_{r_0}$
for some $r_0>0$. Then one has
\begin{enumerate}
\item {\em either } the following {\em transversality} property:
\begin{equation}\label{eqn:Qtran}
D^2g^q(a_0) Q(a_0)\not= Q'(a_0)(\kappa -1);
\end{equation}
\item {\em or persistency of periodic points with the same multiplier holds:} there is a neighborhood $W_1$ of $c_1$ and holomorphic functions $a_j(w)$ defined in $W_1$ with $a_j(c_1)=a_j$ such that for each $w\in W_1$, $G_w^q(a_j(w))=a_j(w)$ and $DG_w^q(a_0(w))=\kappa$ is constant.
\end{enumerate}
\end{mainthm}
In this paper we will not check whether a particular family of maps satisfies the lifting property, but
refer to the results proved in the   second part of \cite{LSvS1,LSvS1a} where the following is shown (with obvious changes in the proof):
\begin{lemma}
Let $f_w$, $U$ be one of the following.
\begin{itemize}
\item $f_w(z)=f(z)+w$, $w\in \C$ and 
$U=D\setminus\{0\}$ with $f\in \mathcal F$ defined in \S\ref{subsec:additive},
\item $f_w(z)=wf(z)$, $w\in\C\setminus\{0\}$, with $f\in \mathcal E\cup \mathcal E_0$ defined in \S\ref{subsec:multiplicative} and $U=D\setminus f^{-1}(1)$ for $f\in \mathcal E$,
$U=D\setminus \pm f^{-1}(1)$ for $f\in \mathcal E_0$.
\end{itemize}
Let $c_1$ be attracted to either hyperbolic attracting or parabolic periodic orbit of $f_{c_1}$ and $W_1$ is a neighborhood of $c_1$.
Then $g=f_{c_1}|_U$ is a marked map, $G_w=f_w|_U$, $w\in W_1$ is a holomorphic deformation of $g$
and $(g,G)_W$ satisfies the lifting property for any $P_{r_0}\subset U$.
\end{lemma}
The above classes of families
include  families of maps with an essential singularity, such as $G_w(z)= b e^{-1/|z|^\ell} +w$ for real non-zero $z$.
A  lifting property also holds within the spaces of rational and polynomial maps, see \cite[Appendix C]{LSvS1a}.

\medskip

\subsection{Clarifying the transversality condition~(\ref{eqn:Qtran}) and the non-degeneracy condition}
\begin{lemma}\label{lem2.2}
\begin{enumerate}
\item If   $\kappa=1$ then (\ref{eqn:Qtran}) implies
 \begin{equation}Q(a_j)\not=0\mbox{ for all }j=0,1,\dots,q-1. \end{equation}
\item If $\kappa\not = 1$ then there exists a holomorphic function $w\mapsto a_0(w)$
for $w$ near $c_1$ so that
$a_0(w)$ is a fixed point of $G_w^q$ for all $w$ near $c_1$ and so that  $a_0(c_1)=a_0$. Defining $\kappa(w)=DG_w^q(a_0(w))$
we have that
 (\ref{eqn:Qtran}) implies
  \begin{equation}\label{eqn:kappa'}
    \kappa'(c_1)=\frac{D^2g^q(a_0) Q(a_0)- Q'(a_0)(\kappa-1)}{1-\kappa}\ne 0 .
    \end{equation}
\end{enumerate}
\end{lemma}
\begin{proof}
 If $\kappa=1$, then  (\ref{eqn:Qtran}) is reduced to $Q(a_0)\not=0$, which is equivalent to $Q(a_j)\not=0$ for all $j=0,1,\ldots, q-1$. If
 $\kappa\not=1$ then  by the Implicit Function Theorem, there exists holomorphic functions $a_j(w)$, defined near $c_1$ such that $a_j(c_1)=a_j$ and  $G_w(a_j(w))=a_{j+1}(w)$ for all $0\le j<q$, where $a_q(w)=a_0(w)$. Let $\kappa(w)=DG_w^q(a_0(w))$.
To see that (\ref{eqn:kappa'}) holds,  let $\mathcal{G}(w,z)=G_w^q(z)$. Then
$$\mathcal{G}(w, a_j(w))=a_j(w), \dfrac{\partial \mathcal{G}}{\partial z} (w, a_j(w))=\kappa(w),$$
for each $j$.
Differentiating and evaluating at $w=c_1$,  we obtain
$$Q(a_j)=(1-Dg^q(a_j)) a_j'(c_1), Q'(a_j)+D^2g^q(a_j)a_j'(c_1)=\kappa'(c_1).$$
Thus the equality in (\ref{eqn:kappa'}) holds. The inequality in (\ref{eqn:kappa'}) is equivalent to  (\ref{eqn:Qtran}).
\end{proof}

\begin{remark} In the parabolic case, the non-degeneracy condition is necessary as shown by the following example. Let $G_w(z)= w\sin z$. Choose $a_0\in (\pi/2, 3\pi/2)$ so that $\tan a_0=-a_0$ and let $w_0=1/\cos a_0$, $g=G_{w_0}$. Then $\mathcal{O}=\{a_0,-a_0\}$ is a cycle of $g$ of period $2$ with $g'(a_0)=g'(a_1)=1$. This parabolic cycle attracts both critical values $w_0$ and $-w_0$ of $g$ and is degenerate. On the other hand,
$$Q(a_0)=\left.\dfrac{d}{dw} G_w^2(a_0)\right|_{w=w_0}=\sin (-a_0)+Dg(-a_0)\sin (a_0)
=0.$$
Note that at the parameter $w_0=1/\cos(a_0)\approx-2.26 $ the family of maps $G_w(x)=w\sin(x)$, $w\in \R$ undergoes
a pitchfork bifurcation, where the attracting period two orbit of this map for $w\in (w_0,w_0+\epsilon)$ becomes
for $w\in (w_0-\epsilon,w_0)$ a repelling two orbit and splits-off two new periodic orbits, both of which
are attracting,  see Figure~\ref{fig1} in the last section.
\end{remark}

\subsection{Applications to transversality within additive complex families}
\label{subsec:additive}
Let us start by  complex families of the form $f_c(z)=f(z)+c$, $c\in \C$.
Let $\mathcal{F}$ denote the collection of holomorphic maps $f: D\to V$, where
\begin{enumerate}
\item [(1)] $D$ is a bounded open set in $\C$ with $0\in \overline{D}$;
\item [(2)] $V$ is a bounded open set in $\C$;
\item [(3)] $f: D\setminus\{0\}\to V\setminus \{0\}$ is an un-branched covering;
\item [(4)] The following separation property holds:  $V\supset B(0;\diam(D)) \supset D$.
\end{enumerate}

\begin{theorem}\label{thm:Main2}
Let $f\in \mathcal{F}$ and let $G_w(z)= f(z)+w$. Suppose that $c_1\in\overline{D}$ is such that $g=G_{c_1}$ has an attracting or parabolic cycle $\mathcal{O}=\{a_0,\cdots,a_{q-1}\}$ with multiplier $\kappa\not=0$. Then $c_n=g^{n-1}(c_1)$ is well defined and converges to $\mathcal{O}$ as $n\to\infty$ and the following transversality holds:
\begin{equation} \kappa'(c_1)\ne 0\mbox{ if } \kappa\ne 1
\,\, \mbox{ and } \,\, Q(a_j)\ne 0\mbox{ for } a_j\in \mathcal{O}\mbox{ if } \kappa=1.\label{eq:trans} \end{equation}
\end{theorem}

\begin{example} The conclusion of the above theorem applies for example for the following families:
\begin{itemize}
\item $f_c(z)=z^d+c$, $c\in \C$ and where $d\ge 2$ is an integer.
\item $f_c(x)= b e^{-1/|x|^\ell} +c$ where
$\ell\ge 1$, $b> 2(e\ell)^{1/\ell}$ are fixed and $c\in D$.
Here $D=U^-\cup U^+$, $U^-=-U^+$, $U^\pm$ are disjoint topological disks symmetric w.r.t. the real axis and $\{0\}=\overline{U^+}\cap \overline{U^-}$. Furthermore,
there is $R>1$ such that $f_0:D\to \D_R\setminus \{0\}$ is an unbranched covering, $D\subset \D_R$, and
$U^+\cap \R\supset (0,\beta]$ where $\beta>0$ is so that the map $f_{-\beta}$ has the Chebysheb combinatorics: $f_{-\beta}(\beta)=\beta$.
\end{itemize}
\end{example}

\begin{remark} When $G_c$ is a real family, the sign of $\kappa'$ and $Q(a_0)$ is given in Section~\ref{sec:realmaps}, see also
Appendix A.
\end{remark}
\begin{remark}\label{quadr}
For  the quadratic family $G_c(z)=z^2+c$  the inequalities in (\ref{eq:trans})  were already known.
The inequality $\kappa'(c_1)\ne 0$ for $c_1$ so that $G_{c_1}$ has a hyperbolic attracting periodic point
was established in  \cite{DH1}.  When $c_1$ is real and $G_{c_1}$ has either a hyperbolic attracting or a parabolic periodic point with multiplier $+1$, the signs for $\kappa'(c_1)$ and $Q(a_0)$
were also already known, see for example
\cite[Lemma 4.5]{Milnor}.
\end{remark}

\subsection{Applications to transversality within multiplicative complex families}
\label{subsec:multiplicative}
To state our next theorem, we say that $v$ is an {\em asymptotic value} of a holomorphic map $f\colon D\to \C$
if there exists a path $\gamma\colon [0,1)\to D$ so that $\gamma(t)\to \partial D$ and $f(\gamma(t))\to v$ as $t\uparrow 1$.
We say that $v$ is a {\em singular value} if it is a critical value or an asymptotic value.
Let us next consider families of the form $f_w (z)=w f(z)$ where
$f: D\to V$ is a holomorphic map  such that:
\begin{enumerate}
\item[(a)] $D,V$ are open sets which are symmetric w.r.t. the real line so that
$f(D)=V$;  
\item[(b)] Let $I=D\cap \R$ then there exists $c>0$ so that $I\cup \{c\}$ is a (bounded or unbounded) open interval and $0\in \overline I$, $c\in \mbox{int}(\overline{I})$.
Moreover,   $f$ extends continuously to $\overline{I}$ so that  $f(I)\subset   \R$ and $\lim_{z\in D, z\to 0} f(z)=0$.
\item[(c)]   Let $D_+$ be the component of $D$ which contains $I\cap (c,\infty)$, where $D_+$ might be equal to $D$.
Then $u\in D\setminus \{0\}$ and $v\in D_+\setminus \{0\}$, $v\ne u$,  implies  $u/v\in V$. 
\end{enumerate}

\noindent
Let $\mathcal{E}$ be the class of maps which satisfy $(a)$,$(b)$,$(c)$ and assumption $(d)$:
\begin{enumerate}
\item[(d)]   $f\colon D\to V$ has no singular values in $V\setminus \{0,1\}$
and $c>0$ is minimal such that $f$ has a positive local maximum at $c$ and $f(c)=1$.
\end{enumerate}
\noindent
Similarly let  $\mathcal{E}_o$ be the class of maps which satisfy $(a)$,$(b)$,$(c)$ and assumption $(e)$:
\begin{enumerate}
\item[(e)] $f$ is odd,   $f\colon D\to V$ has no singular values in $V\setminus \{0,\pm 1\}$
  and $c>0$ is minimal such that $f$ has a positive local maximum at $c$ and $f(c)=1$.
\end{enumerate}
The class $\mathcal{F}$ was introduced in \cite{LSvS1} and classes $\mathcal{E}$, $\mathcal{E}_o$
in \cite{LSvS1,LSvS1a} and include maps for which $V$ or $D$ are bounded sets.

\begin{theorem}\label{thm:Main3}
Let  $f: D\to V$ be holomorphic map from $\mathcal{E}\cup \mathcal{E}_o$ 
and define $g=c_1\cdot f\colon D \to c_1\cdot V$ where we assume that $c_1\in D_+\setminus \{0\}$. 
Assume that $g$ has a hyperbolic attracting or a parabolic cycle $\mathcal{O}=\{a_0,\cdots,a_{q-1}\}\subset D\setminus\{0\}$ with multiplier $\kappa$.
Take $G_w(z)=w f(z)$ where $w\in W:=D_+\setminus \{0\}$. 
Then at $w=c_1$, one of the following holds:
\begin{itemize}
\item  transversality holds: $$ \kappa'(c_1)\ne 0\mbox{ if } \kappa\ne 1
\,\, \mbox{ and } \,\, Q(a_j)\ne 0\mbox{ for } a_j\in \mathcal{O}\mbox{ if } \kappa=1,$$
\item $f\in \mathcal{E}_o$ and $\mathcal{O}$ is symmetric with respect to the origin.
\end{itemize}
\end{theorem}

%

\begin{example} The conclusion of the previous theorem applies for example to the following families:
%
\begin{itemize}
\item $f_b(z)=b z(1-z)$, $b\in \C\setminus \{0\}$,
\item $f_b (z)=b \exp(z) (1-\exp(z))$, $b\in \C\setminus \{0\}$,
\item $f_b (z)=b [\sin(z)]^2$, $b\in \C\setminus \{0\}$,
\item $f_b (z)=b \sin(z)$, $b\in \C\setminus \{0\}$,
 \item $f_b(z)=bf(z)$ where $f$ is the unimodal map $f\colon [0,1]\to [0,1]$ defined by 
$$f(x)=\exp(2^\ell) \left( -\exp(-1/|x-1/2|^\ell) + \exp(-2^\ell) \right)$$
satisfying $f(0)=f(1)=0$, $f(1/2)=1$  which has a
flat critical point at $c=1/2$. Here $\ell\ge \ell_0$ where $\ell_0$ is chosen sufficiently large.
This implies that $f$ has an extension $f\colon D\to V$ which is in  $\mathcal E_0$, where
 $V$ is a punctured bounded disc and $D$ consists of two components $D_-\cup D_+$.
 Here we assume that the parameter $b\in D_+$.
\end{itemize}
\end{example}


\begin{remark}
The classes $\mathcal{E}$ and $\mathcal{E}_o$ both contain maps for which the set of singular values has infinite cardinality.
\end{remark}

\subsection{Periodic points do not disappear after they are born}
For real maps, additional arguments allow us to obtain the sign of $\kappa'(c_1)$ and $Q(a_0)$.

Each $f\in \mathcal{E}\cup \mathcal{E}_o$ defines naturally a unimodal map $f: J:=(0, b)\to\R$ where
$$b=\sup\{b'\in I: b'>0 \text{ and }f(x)>0\mbox{ for }x\in(0,b')\}.$$ We denote by $\mathcal{E}_u$ and $\mathcal{E}_{o,u}$ the collection of unimodal maps obtained in this way from maps in $\mathcal{E}$ and $\mathcal{E}_o$ respectively. Recall that $c$ is the turning point of $f$ in $J$.

We denote by $\mathcal{F}_u^+$ (resp. $\mathcal{F}_u^-$) the collection of $C^1$ unimodal maps $f: J\to \R$, where $J\ni 0$ is an open interval, such that $f|J\setminus \{0\}$ allows an extension to a map $F:D\to V$ in $\mathcal{F}$ with $J\setminus \{0\}=(D\cap \R)\setminus \{0\}$ and such that $f$ has a maximum (resp. minimum) at $0$. Put $c=\{0\}$.

\begin{theorem}\label{thm:eremenko}
Consider a family of unimodal maps $f_t$ satisfying one of the following:
\begin{enumerate}
\item[(i)] $f_t=f+t$ with $f\in \mathcal{F}_u^+$ and $t\in J$;
\item[(ii)] $f_t= t\cdot f$ with $f\in \mathcal{E}_u\cup \mathcal{E}_{o,u}$ and $t\in J$.
\end{enumerate}
Suppose $f_{t_*}$, $t_*>c$, has a period cycle $\mathcal{O}$, then for any $t\in J$ with $t\ge t_*$, $f_t$ has a periodic cycle $\mathcal{O}_t$ of the same period such that $\mathcal{O}_t$ depends on $t$ continuously and $\mathcal{O}_{t_*}=\mathcal{O}$.

Similarly for
\begin{enumerate}
\item [(iii)] $f_t=f+t$ with $f\in\mathcal{F}_u^-$ and $t\in J$,
\end{enumerate}
if $f_{t_*}$, $t_*<c$, has a period cycle $\mathcal{O}$, then for any $t\in J$ with $t\le t_*$, $f_t$ has a periodic cycle $\mathcal{O}_t$ of the same period such that $\mathcal{O}_t$ depends on $t$ continuously and $\mathcal{O}_{t_*}=\mathcal{O}$.
\end{theorem}

%

\section{Outline of the proof of the Main Theorems} \label{sec:outline}
The setting in this paper (and it's purpose) is similar to that in \cite{LSvS1}, \cite{LSvS1a}
except there the case where the postcritical set is
finite is considered. 
Here we will  follow the same strategy in the proof as in that paper. So let us recall the
main steps in the proof of Theorem 2.1 of \cite{LSvS1} or the Main Theorem of \cite{LSvS1a}
:

\begin{enumerate}[leftmargin=8mm]
\item[(A)]  Assume that transversality fails.
Then there exists a holomorphic motion $h_\lambda$ of $P$ over $(\D_\epsilon,0)$ with the speed $v$ at $\lambda=0$:
$$\frac{dh_\lambda(c_n)}{d\lambda}\bv_{\lambda=0}=v(c_n) =\frac{d}{dw}\left. 
G_w^{n-1}(w)\right |_{w=c_1} \mbox{ for $n=1,2,\cdots$.}$$
\item[(B)]  Let $h^{(0)}_\lambda=h_\lambda$ and $h^{(k+1)}_\lambda$ be the lift of $h^{(k)}_\lambda$ for
$k=0,1,2,\cdots$. By the lifting property, all holomorphic motions $h^{(k)}_\lambda$ of $P$ are defined over $(\D_{\hat\epsilon},0)$ for some $\hat\epsilon>0$ and are uniformly bounded.
Moreover, by (A), all $h^{(k)}_\lambda$ are asymptotically invariant of order $m=1$, i.e.,
$$h^{(k+1)}_\lambda(c_n)-h^{(k)}_\lambda(c_n)=O(\lambda^{m+1})$$
with $m=1$. Consider averages
\begin{equation}
\hat h_\lambda^{(k)}(z)= \dfrac{1}{k} \sum_{i=0}^{k-1} h_\lambda^{(i)}(z),\label{eq:av} \end{equation}
$k=1,2,\cdots$ and let $\hat h_\lambda$ be a limit map of $\hat h_\lambda^{(k)}$ along a subsequence.
Then:

\item[(B1)] $\hat h_\lambda$ is again a holomorphic motion of $P$ over (perhaps, smaller) disk,

\item[(B2)] $\hat h_\lambda$ is asymptotically invariant of order $m+1=2$.

\item[(B3)]
 Repeating the procedure, we find that for every $m=1,2,\dots$ there is a holomorphic motion which is asymptotically invariant of order $m$.

 \item[(C)]  Finally, the (B3) yields that the {\lq}critical relation{\rq}
 persists  for all $w$ in a manifold containing $c_1$ of dimension $>0$.
\end{enumerate}
\medskip

When $P$ is a finite set, steps (A) and (B1) are straightforward. In the current set-up this can be
also made to work, as  is shown in this paper, but sometimes with considerable technical efforts. Moreover, we need to require the lifting property to be satisfied not only on the postcritical set $P$ but also on a bigger set which is a local basin of attraction of either hyperbolic or parabolic cycle.

\begin{definition} A holomorphic motion $h_\lambda$ of $P_{r_0}:=P\cup \Omega_{r_0}$
 over $(\D_\eps,0)$ is called {\em admissible} if for each $\lambda\in \D_\eps$, $z\mapsto h_\lambda(z)$ is holomorphic in the interior of $\Omega_{r_0}$. It is called {\em asymptotically invariant of order $m$} if for each $z\in P$,
$$\widehat{h}_\lambda(z)- h_\lambda(z)=O(\lambda^{m+1}) \text{ as } \lambda\to 0$$
where $\widehat{h}_\lambda$ is the lift of $h_\lambda$.
\end{definition}
\medskip
If $h_\lambda$  is asymptotically invariant of order $m$, then: (1) its lift $\widehat{h}_\lambda$ is asymptotically invariant of order $m$, too, and (2) $G_{h_\lambda(c_1)}(h_\lambda(x))=h_{\lambda}(g(x))+O(\lambda^{m+1})$, $x\in P$. See \cite{LSvS1a} for details.

\newtheorem*{ThmA}{Theorem A}
\newtheorem*{ThmB}{Theorem B}
\newtheorem*{ThmC}{Theorem C}

The proof of the Main Theorem is broken into the following three steps.
\begin{ThmA}
Assume transversality fails, so assume equality holds in (\ref{eqn:Qtran}).
Then there exists an admissible holomorphic motion $H_\lambda$ of the set $P_{r_0}$ over $(\D_\epsilon,0)$ for some $\epsilon>0$ such that
$$\frac{dH_\lambda}{d\lambda}\bv_{\lambda=0}(c_n)=v(c_n), \ n=1,2,\cdots,$$
where $v(c_n)=\left.\frac{d}{dw}G_w^{n-1}(w)\right|_{w=c_1}.$
In particular, the holomorphic motion is asymptotically invariant of order $1$.
\end{ThmA}

\begin{ThmB} For any $m\ge 1$, if there is an admissible holomorphic motion $h_{\lambda}$ of the set $P_{r_0}$ over some $\D_{\eps}$ which is asymptotically invariant of order $m$, then there is an admissible holomorphic motion $\widetilde{h}_{\lambda}$ of the set $P_{r_0}$ over some $\D_{\eps'}$ which is asymptotically invariant of order $m+1$ such that
$$\widetilde{h}_\lambda(z)-h_\lambda(z)=O(\lambda^{m+1})\text{ as } \lambda\to 0, \text{ for any } z\in P.$$
\end{ThmB}

\begin{ThmC}
Suppose for any $m\ge 1$, there is an admissible holomorphic motion $h_{\lambda,m}$ of $\overline{P}_{r_0}$ over $(\D_{\eps_m},0)$ for some $\eps_m>0$ such that $$\dfrac{d}{d\lambda}h_{\lambda,m}(c_1)\bv_{\lambda=0}=1$$
and $h_{\lambda,m}$ is asymptotically invariant of order $m$.
Then the second alternative of the Main Theorem holds.
\end{ThmC}

\begin{proof} [Proof of the Main Theorem] Assume that the transversality condition fails. Then by Theorems A and B, we obtain a sequence of admissible holomorphic motions $h_{\lambda,m}$ of $\overline{P}_{r_0}$ satisfying the assumption of Theorem C. Thus the second alternative of the Main Theorem holds.
\end{proof}
Theorems A and B will be proved in Sect \ref{sec:ThmA} and \ref{sec:ThmB} respectively, where the hyperbolic case is much easier and will be done first. Theorem C will be proved in Sect \ref{sec:ThmC}.

\subsection{How to construct admissible holomorphic motions}
We end this section with the following lemma which is useful in constructing admissible holomorphic motions.

\begin{lemma} \label{lem:improvemotion}
Let $h_\lambda$ be a holomorphic motion of $P_{r_0}$ over $(\D_\eps,0)$ for some $\eps>0$ which is asymptotically invariant of order $m$. Assume that for each $K_0>1$, there is a $g$-invariant open set $W\subset P_{r_0}$ such that
\begin{enumerate}
\item $z\mapsto h_\lambda(z)$ is $K_0$-qc in $W$ for all $\lambda\in \D_\eps$;
\item $\Omega_{r_0}\subset \bigcup_{n=0}^\infty g^{-n}(W)$.
\end{enumerate}
Then there exists an admissible  holomorphic motion $\check{h}_\lambda$ of $P_{r_0}$ over $(\D_{\eps'},0)$ for some $\eps'>0$ such that
\begin{equation}\label{eqn:hwidehath}
h_{\lambda}(c_k)-\check{h}_{\lambda}(c_k)=O(\lambda^{m+1})\text{ for each }k\ge 1.
\end{equation}
\end{lemma}

\begin{proof}
By the lifting property, restricting to a smaller disk $\D_{\eps'}$, the holomorphic motion $h_\lambda$ allows successive lifts $h_\lambda^{(n)}$ of $P_{r_0}$ over $\D_{\eps'}$.
By compactness of holomorphic motions, there exists $n_k\to\infty$, such that
$h_\lambda^{(n_k)}$ converges to a holomorphic motion $\check{h}_\lambda$ of $P_{r_0}$ over $\D_{\eps'}$ locally uniformly.

For each $k\ge 1$, by asymptotic invariance of $h_\lambda$, $h_\lambda^{(n)}(c_k)-h_\lambda(c_k)=O(\lambda^{m+1})$, hence (\ref{eqn:hwidehath}) holds.

Let us prove that $z\mapsto \check{h}_\lambda(z)$ is holomorphic in $\Omega_{r_0}^o$. To this end, it suffices to show that $\check{h}_\lambda$ is $K_0$-qc in $\Omega_{r_0}^o$ for each $K_0>1$. Given $K_0>1$, let $W$ be given by the assumption.
For each $z_0\in \Omega_{r_0}^o$, there is a neighborhood $Z$ of $z_0$ and $n_0\ge 1$ such that $g^{n_0}(Z)\subset W$, and hence $g^n(Z)\subset W$ for all $n\ge n_0$.  Since
$$h_\lambda^{(0)}(g^n(z))= G_{h_\lambda^{(0)}(c_1)}\circ \cdots  \circ G_{h_\lambda^{(n-1)}(c_1)} (h_\lambda^{(n)}(z)),$$
it follows that $h_\lambda^{(n)}$ is $K_0$-qc in $Z$ for each $n\ge n_0$. Therefore, $\check{h}_\lambda$ is $K_0$-qc in $Z$.
\end{proof}

\section{Admissible holomorphic motions of asymptotic invariance order one}\label{sec:ThmA}
In this section, we shall prove Theorem A. Let $$v(c_n)=\left.\frac{d}{dw} G_w^{n-1}(w)  \right|_{w=c_1},\,\, n\ge 1.$$
So we have
\begin{equation}\label{vrec}
v(c_{n+1})=L(c_n)+Dg(c_n) v(c_n), \, n\ge 1,
\end{equation}
where
\begin{equation} \label{eqn:Lz}
L(z)=\left.\frac{\partial G_w(z)}{\partial w}\right|_{w=c_1}.
\end{equation}
Differentiating (\ref{liftdef}) at $\lambda=0$ we see from (\ref{vrec}) that if $H_\lambda$ is a holomorphic motion of $P_r$ (or simply $P=\{c_n\}\subset P_r$) with $d H_\lambda/d\lambda|_{\lambda=0}(c_n)=v(c_n)$, $n\ge 1$, then $H_\lambda$ is asymptotically invariant of order $1$.

Below we will use the following formula (which follows immediately by induction):
\begin{equation}
Q(z):=\dfrac{\partial G^q_w}{\partial w}\bv_{w=c_1} (z) =\sum_{i=0}^{q-1} Dg^i(g^{q-i}(z)) L(g^{q-i-1}(z)).
\label{eq:partGL}
\end{equation}

\subsection{The hyperbolic case}


The following lemma is essentially contained in \cite[Lemma 6.10 and Remark 6.7]{ALdM}, but we add the proof for the reader's convenience.

\begin{lemma}\label{lem:coho}
Let $f:\D\to \D$ be a holomorphic injection such that $f(0)=0$ and $\kappa=f'(0)\in \D\setminus \{0\}$ and let $\Gamma:\D\to \C$ be a holomorphic function. Let $a\in \D\setminus \{0\}$ and $b\in \C$ be arbitrary.
Assume that
\begin{equation}\label{eqn:hypcoh}
\Gamma(0)f''(0)-\Gamma'(0)(f'(0)-1)=0.
\end{equation}
Then there exists a holomorphic map $w:\D\to \C$ such that $w(a)=b$ and
\begin{equation}\label{eqn:coh}
w\circ f(z)= \Gamma(z)+ f'(z) w(z).
\end{equation}
\end{lemma}

\begin{proof} 
Let $\varphi: \D \to \C$ denote the Koenigs linearization, i.e., $\varphi$ is a conformal map onto its image, with $\varphi(0)=0$, $\varphi'(0)=1$ and $\varphi (f(z))=\kappa \varphi(z)$ for all $z\in \D$. Write $\widetilde{w}(z)=w\circ \varphi^{-1}(z)\varphi'(\varphi^{-1}(z))$, and $\widetilde{\Gamma}(z)=\Gamma(\varphi^{-1}(z)) \varphi'(f\circ \varphi^{-1}(z))$. Since $\varphi'(f(z))f'(z)=\kappa \varphi'(z)$, the equation (\ref{eqn:coh}) is reduced to the following form:
\begin{equation}\label{eqn:coh1}
\widetilde{w}(\kappa z)=\widetilde{\Gamma}(z)+ \kappa \widetilde{w}(z).
\end{equation}
From $\varphi(f(z))=\kappa \varphi(z)$, we obtain
$$\varphi'(f(z))f'(z)=\kappa \varphi'(z)$$ and
$$\varphi''(f(z))f'(z)^2+\varphi'(f(z)) f''(z) =\kappa \varphi''(z),$$
hence $\varphi''(0)=f''(0)/(\kappa-\kappa^2)$. Thus the condition (\ref{eqn:hypcoh}) is equivalent to $\widetilde{\Gamma}'(0)=0$.
Under this condition,
$$u(z)=-\kappa^{-1}\sum_{n=0}^\infty \widetilde{\Gamma}'(\kappa^n z)$$
defines a holomorphic map satisfying $\kappa u(\kappa z) =\widetilde{\Gamma}'(z)+\kappa u(z).$
Let $\widetilde{w}$ be a holomorphic map such that $\tilde{w}(\varphi(a))=\varphi'(a)b$, 
$\tilde{w}(0)=\tilde{\Gamma}(0)/(1-\kappa)$ and such that 
$\tilde{w}''(z)=u'(z)$. 
Then it solves the equation (\ref{eqn:coh1}) and $w(z)=\widetilde{w}\circ \varphi(z)/\varphi'(z)$ solves the equation (\ref{eqn:coh}) with $w(a)=b$. 
%
\end{proof}
\begin{proof}[{\bf Proof of Theorem A in the attracting case}]
Let $\delta>0$ be such that $g^q$ maps $B(a_0, \delta)$ injectively into $B(a_0,\delta)$ and let $N$ be such that $c_N\in B(a_0,\delta)$. By assumption,
$$Q(a_0) D^2 g^q (a_0)=Q'(a_0)(\kappa-1).$$
So by Lemma~\ref{lem:coho}, there is a holomorphic map $w: B(a_0,\delta)\to\C$ such that
$$w(g^q(z))=Q(z)+ Dg^q(z) w(z) \text{ for } z\in B(a_0,\delta),$$
and $w(c_N)=v(c_N)$. The function $w$ extends naturally to a map $w: P_{r_0}\to\C$ that satisfies $w(g(z))=L(z)+Dg(z) w(z)$ which is holomorphic in $\Omega_{r_0}^o$. Since $w(c_N)=v(c_N)$, and $v(c_{n+1})=L(c_n)+Dg(c_n) v(c_n)$, it follows that $w(c_n)=v(c_n)$ for all $n\ge N$.

%

In particular, $w|P_{\delta/2}$ is Lipschitz. Thus $H_\lambda(z):= z+\lambda w(z)$ defines a holomorphic motion of $P_{\delta/2}$ over $(\D_{\eps},0)$ for some $\eps>0$. Since every point in $\Omega_{r_0}$ eventually lands in $\bigcup B(a_j,\delta/2)$, applying Lemma~\ref{lem:improvemotion} completes the proof. 
\end{proof}

\subsection{The parabolic case}
\begin{lemma}\label{prop:vC1}
Let $W$ be a neighborhood of $0$ in $\C$ and let $f, \Gamma: W\to \C$ be holomorphic functions with $f(0)=0$, $f'(0)=e^{2\pi i l/p}$ and $D^{p+1}f^p(0)\not=0$, where $l, p\in \Z$, $p\ge 1$ and $(l,p)=1$, and with
\begin{equation}\label{eqn:QQ'tan}
\Gamma'(0)(f'(0)-1)= \Gamma(0) f''(0).
\end{equation}
Let $P=\{z_n: n\ge 1\}\subset W$ be an infinite orbit of $f$ such that $z_n=f^{n-1}(z_1)\to 0$ and let $v: P\to \C$ be a function such that
$$v(z)f'(z)+ \Gamma(z)= v(f(z)), \mbox{ for each } z\in P.$$
Then $v$ extends to a $C^1$ map $V: \C\to \C$ such that $\overline{\partial} V(0)=0$.
\end{lemma}
\begin{proof}
{\bf Step 1.} Let us prove that there exists a polynomial $h$ and a holomorphic function $\widehat{\Gamma}$ defined near $0$ such that
\begin{equation}\label{eqn:boldQ}
D^j\widehat{\Gamma}(0)=0, \text{ for } j=0,1,\cdots, p+1,
\end{equation}
and such that
\begin{equation}\label{eqn:step1}
\textbf{v}(z_n)(f^p)'(z_n)+\widehat{\Gamma}(z_n)=\textbf{v}(z_{n+p}),
\end{equation}
for all $n$ large enough, where $\textbf{v}(z)= v(z)+h(z)$.

We first deal with the case $p=1$. Then $f''(0)\not=0$ and $\Gamma(0)=0$. Define
$h_1(z)=\Gamma'(0)/f''(0)$, $\Gamma_1=\Gamma(z)-h_1(z)f'(z)+h_1(f(z))$ and
$v_1(z)=v(z)+h_1(z)$. Then $\Gamma_1(0)=\Gamma_1'(0)=0$ and $v_1(z)f'(z)+\Gamma_1(z)=v_1(f(z))$ for each $z\in P$. Let now $b_1=\Gamma_1''(0)/f''(0)$, $h(z)=h_1(z)+b_1 z$. Then
$\widehat{\Gamma}(z)=\Gamma_1(z)-b_1(z f'(z)-f(z))$ and
$\textbf{v}(z)=v_1(z)+b_1 z=v(z)+h(z)$ satisfy the desired property.

Now assume $p>1$.   {\bf Claim.} For any $1\le k\le p$, there is a polynomial $h_{1,k}$ such that
$\Gamma_{1,k}(z)=\Gamma(z)- h_{1,k}(z) f'(z) + h_{1,k}(f(z))$ satisfies
$$\Gamma_{1,k}(z)=O(z^{k+1}).$$
Let us prove this by induction. For $k=1$, define $h_{1,1}(z)=\Gamma(0)/(f'(0)-1)$. Then the claim follows from (\ref{eqn:QQ'tan}).
Assume now the claim holds for some $1\le k<p$. Let $A$ be such that $\Gamma_{1,k}(z)=Az^{k+1}+O(z^{k+2})$.
Define $h_{1,k+1}(z)=h_{1,k}(z)+ b z^{k+1}$, where
$b=A/(f'(0)-f'(0)^{k+1})$. Then
$$\begin{array}{rl}\Gamma_{1,k+1}(z)&=\Gamma_{1,k}(z)-b z^{k+1}f'(z)+b f(z)^{k+1}\\
&=Az^{k+1} -b f'(0)z^{k+1}+ bf'(0)^{k+1} z^{k+1}+O(z^{k+2})=O(z^{k+2})\end{array}$$
completing the proof of the claim.


Define $h_1=h_{1,p}$,
$\Gamma_1(z)=\Gamma_{1,p}$,  $v_1(z)=v(z)+ h_1(z) $ and
$$\widetilde{\Gamma}_1(z)=Df^p(z) \sum_{k=1}^{p} \frac{\Gamma_1(f^{k-1}(z))}{Df^k(z)}=O(z^{p+1}).$$ Then $v_1(z)f'(z)+\Gamma_1(z)=v_1(f(z))$ for each $z\in P$, hence
$$v_1(z) (f^p)'(z) + \widetilde{\Gamma}_1(z) = v_1(f^p(z)), z\in P.$$
Now take $b_1:=\widetilde{\Gamma}_1^{(p+1)}(0)/(pD^{p+1}f^p(0))$,
$h(z)=h_1(z)+b_1 z$,
$\textbf{v}(z)= v_1(z) +b_1 z$ and $\widehat{\Gamma}=\widetilde{\Gamma}_1(z)-b_1(z(f^p)'(z)-f^p(z))$. 
Then
$\widehat{\Gamma}$ and $\textbf{v}$ satisfy (\ref{eqn:boldQ}) and (\ref{eqn:step1}).

{\bf Step 2.}
Let us prove that there exist $M>0$ and $\sigma>0$ such that
\begin{equation}\label{step2ineq}
|\textbf{v}(z_n)-\textbf{v}(z_{n+p})|\le M|z_n-z_{n+p}|^{1+\sigma}.
\end{equation}
Let $A=D^{p+1}f^p(0)/(p+1)!\not=0$. By the Leau-Fatou Flower Theorem, see Appendix B,
$|(f^p)'(z_n)|\sim 1-|A|(p+1)|z_n|^p$ and
\begin{equation}\label{step2assy}
\frac{z_n}{z_{n+p}}\sim \frac{1}{1-|A||z_n|^p}.
\end{equation}
Fix an arbitrary $\eps\in (0,1)$ and $\delta\in (0, |A|(p-\eps))$. There exists $n_0$ such that
$$\gamma_n:=\left(\frac{|z_n|}{|z_{n+p}|}\right)^{1+\eps} |(f^p)'(z_n)|\le 1-\delta |z_n|^p$$
holds for all $n\ge n_0$. Write $w_n= |\textbf{v}(z_n)|/|z_n|^{1+\eps}$ and let $C_0$ be a constant such that
$|\widehat{\Gamma}(z_n)|\le C_0|z_{n+p}|^{p+2}$ for all $n$ which exists by (\ref{eqn:boldQ}, \ref{step2assy}. Then for all $n\ge n_0$,
\begin{align*}
w_{n+p} \le\gamma_n w_n+\frac{|\widehat{\Gamma}(z_n)|}{|z_{n+p}|^{1+\eps}}
 \le (1-\delta|z_n|^p) w_n + C_0 |z_{n+p}|^{p+1-\eps}.
\end{align*}
Let $M_0>0$ be such that $w_n\le M_0$ for all $n\le n_0$ and such that $C_0|z_{n+p}|^{1-\eps} <\delta M_0$ for all $n\ge n_0$. Then by induction using that $|z_{n+p}|<|z_n|$, we obtain $w_n\le M_0$ for all $n$.

Finally
\begin{multline*}
\textbf{v}(z_{n+p})-\textbf{v}(z_n)=\widehat{\Gamma}(z_n)+ \left((f^p)'(z_n)-1\right) \textbf{v}(z_n)=O(z_n^{p+2}) +O(z_n^{p+1+\eps})\\=O(z_n^{p+1+\eps})
=O(|z_n-z_{n+p}|^{(p+1+\eps)/(p+1)})=O(|z_n-z_{n+p}|^{1+\eps/(p+1)}).
\end{multline*}

Now, let us prove that $\textbf{v}:P\to\C$ extends continuous to $0$ by  showing that
$$\lim_{n\to\infty}\textbf{v}(z_n)=0.$$
It follows from (\ref{step2ineq}) that given $n$ there exists $\lim_{j\to\infty}\textbf{v}(z_{n+jp})$
so it is enough to prove that the latter limit is $0$ for any $n$. Assuming by a contradiction that this is not the case for some $n_0$  and using (\ref{eqn:boldQ}) and (\ref{eqn:step1}) we get (replacing if necessary $n_0$ by $n_0+j_0p$ with a big $j_0$),
$$|\textbf{v}(z_{n_0+jp})|\le |\textbf{v}(z_{n_0})|\Pi_{i=1}^{j-1}|(f^p)'(z_{n_0+ip})|(1+O(|z_{n_0+ip}|^{p+2}))\to 0$$
as $j\to\infty$ because $\Pi_{i=1}^{j-1}(f^p)'(z_{n_0+ip})=(f^{jp})'(z_{n_0})\to 0$ and because  $\sum_{i=0}^\infty|z_{n_0+ip}|^{p+2}\le\sum_{i=0}^\infty O(i^{-(p+2)/p})<\infty$,
a contradiction.

{\bf Step 3.} We shall now prove that $\textbf{v}: P\to \C$ extends to a $C^1$ function $\textbf{V}:\C\to\C$ such that $\partial \textbf{V}(0)=\overline{\partial } \textbf{V}(0)=0$. Once this is proved, we obtain a desired extension $V$ of $v$ by setting $V=\textbf{V}-h$.

Indeed, by the Leau-Fatou Flower Theorem (see Appendix B)  for each $j=0,1,\ldots, p-1$, there exists $\theta_j\in \R$ with $Ae^{2\pi i\theta_j}=-|A|<0, \theta_{j+1}=\theta_j+2\pi/p$ such that the argument $z_{np+j}$ converges to $\theta_j$ as $n\to\infty$. Moreover, the argument of $z_{n+p}-z_n=Az_n^p (1+o(z_n))$ converges to $\pi$. Therefore, there is a $C^1$ diffeomorphism $H:\C\to \C$ with $H(z)=z+o(|z|)$ near $z=0$ such that $H^{-1}(z_{np+j})$ lies in order on the ray $\theta=\theta_j$. Write $z'_n=H^{-1}(z_n)$. Then $\textbf{v}\circ H(z'_n)-\textbf{v}\circ H(z'_{n+p})=o(|z'_n-z'_{n+p}|)$, $\textbf{v}\circ H|H^{-1}(P)$, and hence $\textbf{v}$, extends to a $C^1$ map defined on $\C$ with zero partial derivatives at $0$.
\end{proof}
\begin{proof}[{\bf Proof of Theorem A in the parabolic case}] Define $v(c_n)=\left.\frac{d G_w^{n-1}(w)}{dw}\right|_{w=c_1}$. By Proposition~\ref{prop:vC1} applying for $f(z)=g^q(z+a_j)$ and $\Gamma(z)=Q(z+a_j)$ for each $j$, the vector field $v$ on $P$ extends to a $C^1$ vector field $V$ in $\C$ such that $\overline{\partial} V(z)\to 0$ as $z\to a_j$, $j=0,1,\ldots, q-1$. We can surely make the extension compactly supported. So $\mu=\overline{\partial} V$ is a qc vector field. By the Measurable Riemann Mapping Theorem, there is a holomorphic motion $h_\lambda$ of $\C$ over some disk $\D_\eps$, such that $\overline{\partial}h_\lambda=\lambda \mu \partial h_\lambda$ and $h_\lambda(z)=z+o(1)$ as $z\to\infty$. In particular, $h_\lambda$ defines a holomorphic motion of $P_{r_0}$ which is asymptotically invariant of order $1$ and
$$\left.\frac{d}{d\lambda} h_\lambda(z) \right|_{\lambda=0}=V(z), \forall z\in P.$$
To complete the proof, we shall apply Lemma~\ref{lem:improvemotion}. Let us verify the conditions. For each $K>1$, choose $r$ small enough so that $|\overline{\partial} V|<(K-1)(K+1)$ holds in $\Omega_r\subset \bigcup_j \overline{B(a_j, r)}$ and let $W=\Omega_r$. Both conditions are clearly satisfied.
\end{proof}

\section{Averaging and promoting asymptotic invariance}\label{sec:ThmB}
In this section, we prove Theorem B.

\subsection{The averaging process}\label{subsec:averaging}
Suppose that $(g,G)_W$ is a local holomorphic deformation of a marked map $g: U\to \C$ which has the lifting property of some set $K$ with $P\subset K$ and $g(K)\subset K$.
Let $h_\lambda$ be a holomorphic motion of $K$ over $(\D_\epsilon,0)$. By the lifting property, there is $\eps'>0$ and a sequence $h_\lambda^{(k)}$, $k=0,1,\cdots$ of holomorphic motions of $K$ over $(\D_{\epsilon'},0)$ so that $h^{(0)}_\lambda=h_\lambda$ and $h^{(k+1)}_\lambda$ is the lift of $h^{(k)}_\lambda$, for each $k\ge 0$.

Let $H_\lambda$ be a (locally uniform) limit for some subsequence $k_i\to \infty$ of
$$\tilde{h}_\lambda^{(k)}:= \dfrac{1}{k} \sum_{i=0}^{k-1} h_\lambda^{(i)}.$$
The following proposition is proved in \cite[Lemma 2.12]{LSvS1}, \cite[Lemma 6.4]{LSvS1a}.
     \begin{prop}\label{mtom+1} Assume
     $h_\lambda^{0}$ is asymptotically invariant of some order $m$, i.e.
     $$h_\lambda^{(k+1)}(c_j)- h^{(k)}_\lambda(c_j)=O(\lambda^{m+1}), \ \ j=1,2,\cdots \mbox{ as }\lambda\to 0$$
     for $k=0$ (hence, for all $k$).
     Then $H_\lambda$ is asymptotically invariant of order $m+1$, i.e.
      $$\hat{H}_\lambda(c_j)- H_\lambda(c_j)=O(\lambda^{m+2}), \ \ j=1,2,\cdots \mbox{ as }\lambda\to 0.$$
     \end{prop}

When $K$ is finite, then $H_\lambda$, when restricting $\lambda$ to a smaller disk, is automatically a holomorphic motion. However, this is not necessarily the case when $K$ has infinite cardinality. We solve this issue by considering holomorphic motions which are `almost' conformal near $\mathcal{O}$, using distortion estimates.


\subsection{The attracting case}
\begin{proof}[{\bf Proof of Theorem B in the attracting case}]
Let $h_\lambda$ be an admissible holomorphic motion of $P_{r_0}$ over $(\D_\eps,0)$ which is asymptotically invariant of order $m$ and let $h_\lambda^{(k)}$, $\hat{h}_\lambda^{(k)}$ and $H_\lambda$ be as in Subsection~\ref{subsec:averaging}.

Let us prove that there is $r\in (0,r_0)$ and $\eps_1>0$ such that $H_\lambda$ is an admissible holomorphic motion of $P_{r}$ over $(\D_{\eps_1},0)$. Indeed, by definition of the lifting property, there exists $M>0$ and $\eps_2>0$ such that $|h_\lambda^{(k)}(z)|\le M$ for all $z\in P_{r_0}$ and $\lambda\in \D_{\eps_2}$. By Slodowski's theorem, there exists $M'>0$ such that $h_\lambda^{(k)}$ extends to a holomorphic motion of $\C$ over $\D_{\eps_2}$ such that $h_\lambda^{(k)}(z)=z$ whenever $|z|>M'$. By Bers-Royden's Theorem~\cite{BersRoyden}, there exists $K(\lambda)>1$ for each $\lambda\in \D_{\eps_2}$ with $K(\lambda)\to 1$ as $\lambda\to 0$ such that for each $k=0,1,\ldots$, $h_\lambda^{(k)}$ is $K(\lambda)$-qc. Thus for each $\delta>0$ there exists $\eps(\delta)>0$ such that $|h_\lambda^{(k)}(z)-z|\le \delta$ for all $z\in P_{r_0}$ and $\lambda\in \D_{\eps(\delta)}$. Since $a_j$ is in the interior of $P_{r_0}$, it follows that there is $\eps_3>0$ such that $|(h_\lambda^{(k)})'(a_j)-1|<\frac{1}{3}$ for all $\lambda\in \D_{\eps_3}$. By
the Koebe Distortion Theorem, there exists $r\in (0,r_0)$ such that for any $z_1, z_2\in B(a_j, r)$, $z_1\not=z_2$, $$\left|\frac{h_\lambda^{(k)}(z_1)-h_\lambda^{(k)}(z_2)}{z_1-z_2}-1\right|<\frac{1}{2},$$
which implies that
$$\left|\frac{H_\lambda(z_1)-H_\lambda(z_2)}{z_1-z_2}-1\right|\le\frac{1}{2},$$
hence $H_\lambda$ is a holomorphic motion of $\Omega_{r}$ over $\D_{\eps_3}$. As $P\setminus \Omega_{r}$ is finite, the statement follows by choosing $\eps_1$ sufficiently small.

To complete the proof, we extend $H_\lambda$ to a holomorphic motion of $P_{r_0}$ in an arbitrary way and then apply Lemma~\ref{lem:improvemotion} as follows: simply take $W=\Omega_r$ for each $K>1$.
\end{proof}
\subsection{The parabolic case}
The parabolic case is more complicated. We shall need the following distortion lemma:
\begin{lemma}\label{lem:distortionparabolic}
Given a positive integer $\m$, $\alpha\in (0,1)$, $M>0$ and $R>0$ with $3R<M$ and $(3R)^\alpha<\pi/(4\m)$,  there is $K_0>1$
such that if $H:B(0,M)\to B(0,M)$ is a $K_0$-qc map satisfying
$H|\partial B(0,M)=id$ and
$$\overline{\partial}H=0\text{ a.e. on }B(0,3R)\setminus \cC(R),$$
where $$\mathcal{C}(R)=\{re^{2\pi it}: 0<r<3R, |t-(2k+1)\pi/\m|< r^{\alpha}\text{ for some } k=0,1,\ldots, \m-1\},$$
then for any $z_1, z_2\in \cC'(R)=\{re^{2\pi it}: 0<r<R, |t-2k\pi/\m|< r^{\alpha}\text{ for some } k=0,1,\ldots, \m-1\}$, we have
$$|H(z_1)-H(z_2)-(z_1-z_2)|\le \dfrac{1}{2}|z_1-z_2|.$$
\end{lemma}
\begin{proof} Choose $q'\in (1, 1+\alpha/2)$ and let $p'>1$ be such that $1/p'+1/q'=1$. Let $D=B(0, 3R)$. Let $\eps>0$ be a small constant to be determined. Since $H|\partial B(0,M)=id$, it is well-known, see for example~\cite[Chapter V]{Ah}, that provided that $K_0$ is sufficiently close to $1$, $$\int\int_{D} |\overline{\partial} H|^{p'} d|u|^2<\eps^{p'},\, \text{ and }|H(z)-z|<\eps R, \text{ for all }z\in D.$$
Since $H$ is ACL, we can apply the Cauchy-Pompeiu Formula
$$H(z)-z=\dfrac{1}{2\pi i} \int_{\partial D} \frac{H(u)-u}{u-z} \, du - \frac{1}{\pi} \int\int_{D} \frac{\overline{\partial} H(u)}{u-z} \,
|du|^2,$$
for $z\in D$.
For $z_1, z_2\in \cC'(R)$, and $u\in \partial D$, we have $|u-z_1|\ge 2R$, $|u-z_2|\ge 2R$, so
\begin{multline*}
 \left|\dfrac{1}{2\pi i} \int_{\partial D} \frac{H(u)-u}{u-z_1} du-\dfrac{1}{2\pi i} \int_{\partial D} \frac{H(u)-u}{u-z_2} du \right|
=  \left|(z_1-z_2)\dfrac{1}{2\pi i} \int_{\partial D} \frac{H(u)-u}{(u-z_1)(u-z_2)} du\right|\\
\le  |z_1-z_2|\dfrac{1}{2\pi} \frac{\eps R}{4R^2}\cdot 2\pi \cdot 3R < \frac{1}{4} |z_1-z_2|,
\end{multline*}
where we choose $\eps$ small enough to obtain the last inequality.
For $u\in \cC(R)$, we have $|u-z_j|\ge \rho |u|$, where $\rho=\rho(\m,\alpha,R)>0$ is a constant. Thus
\begin{multline*}
\int\int_{\cC(R)} \frac{1}{|u-z_1|^{q'}|u-z_2|^{q'}} d|u|^2
\le  \frac{1}{\rho^{2q'}} \sum_{k=0}^{\m-1} \int_0^{3R} \int_{|t-(2k+1)\pi/\m|< r^{\alpha}} \frac{1}{r^{2q'}} r dt dr\\
=  \frac{2\m}{\rho^{2q'}}\frac{(3R)^{2+\alpha-2q'}}{2+\alpha-2q'}=: C.
\end{multline*}
Therefore, 
\begin{multline*}
\left|\int\int_{D} \frac{\overline{\partial} H(u)}{u-z_1} d|u|^2-\int\int_{D} \frac{\overline{\partial} H(u)}{u-z_2}d|u|^2\right|
= \left|\int\int_{\cC(R)} \frac{\overline{\partial} H(u)}{u-z_1} d|u|^2-\int\int_{\cC(R)} \frac{\overline{\partial} H(u)}{u-z_2} d|u|^2\right|\\
=  |z_1-z_2| \left|\int\int_{\cC(R)} \frac{\overline{\partial} H(u)}{(u-z_1)(u-z_2)}d|u|^2\right|\\
\le |z_1-z_2|\left(\int\int_{\cC(R)} |\overline{\partial} H(u)|^{p'} d|u|^2\right)^{1/p'} \left(\int\int_{\cC(R)} \frac{1}{|u-z_1|^{q'}|u-z_2|^{q'}} d|u|^2  \right)^{1/q'}\\
\le  |z_1-z_2| \eps C^{1/q'} <\pi |z_1-z_2|/4,
\end{multline*}
where, once again, we choose $\eps>0$ small enough to obtain the last inequality.
The lemma follows.
\end{proof}
\begin{proof}[{\bf Proof of Theorem B in the parabolic case}] Extend each $h_\lambda^{(k)}$ to a holomorphic motion of $\C$ over $(\D_{\eps'},0)$ and such that $h_{\lambda}^{(k)}(z)=z$ for all $\lambda\in \D_{\eps'}$, $k$ and $|z|>M'$. Fix $\alpha\in (0,1)$.
Let $R=\tau(r_0)/3$ be given by Lemma~\ref{lem:leau-fatou} (1) and let $\tilde K=K_0(p,\alpha, M', R)$ be given by Lemma~\ref{lem:distortionparabolic}. By \cite{BersRoyden}, there exists $\eps_1>0$ such that $h_{\lambda}^{(k)}$ is $\tilde{K}$-qc for all $\lambda\in \D_{\eps_1}$. Thus for each $z_1,z_2\in \cC'_j$, $0\le j<p$, $z_1\not=z_2$, and any $k\ge 0$, $\lambda\in \D_{\eps_1}$,
$$\left|\frac{h_\lambda^{(k)}(z_1)-h_\lambda^{(k)}(z_2)}{z_1-z_2}-1\right|\le \frac{1}{2},$$
hence $$\left|\frac{H_\lambda(z_1)-H_\lambda(z_2)}{z_1-z_2}-1\right|\le \frac{1}{2}.$$
It follows that $H_\lambda$ is a holomorphic motion of $P\cup \bigcup_j \cC'_j(\tau)$ over $\D_{\eps_2}$ of asymptotic invariance of order $m+1$. By extending it to a holomorphic motion of $P_{r_0}$ in an arbitrary way and applying Lemma~\ref{lem:improvemotion} by taking $W=\bigcup \cC'_j(\tau)$ for all $K_0>1$, we complete the proof. 
\end{proof}

\section{Asymptotic invariance of an arbitrarily large order}\label{sec:ThmC}
In this section, we shall prove Theorem C. For each $m\ge 1$, let $h_{\lambda, m}$ be a holomorphic motion of $\overline{P}$ over $(\D_{\eps_m},0)$ for some $\eps_m>0$ which is asymptotically invariant of order $m$ such that $h_{\lambda, m}(c_1)=c_1+\lambda+O(\lambda^2)$ as $\lambda\to 0$.
{\bf Claim:} we may assume that the holomorphic motion $h_{\lambda,m}$ satisfies 
 $h_{\lambda, m}(c_1)=c_1+\lambda$. Indeed,  for each $m$ take a reparametrisation $\lambda=\lambda_m(\mu)$ so that  for
$\hat h_{\mu,m}:=h_{\lambda_m(\mu),m}$ we have $\hat h_{\mu, m}(c_1)=c_1+\mu$. Then $\hat h_{\mu,m}$
is still asymptotically invariant of order $m$.  That is, $G_{\hat h_\mu(c_1)}(\hat h_\mu(z))=\hat h_\mu(g(z))+O(\mu^{m+1})$.
Renaming the new holomorphic motion again $h_{\lambda,m}$  the claim follows.
%
%

Then for each $k\ge 1$,
$$h_{\lambda,m} (c_k) =G_{\lambda+c_1}^{k-1}(\lambda+c_1)+O(\lambda^{m+1})$$
and so the first $m$ terms in the Taylor series of $h_\lambda(c_k)$ is fixed according to this formula and therefore
$h_{\lambda,m}(c_k)-h_{\lambda,{m+1}}(c_k)=O(\lambda^{m+1})$.
Assume without loss of generality that $c_{nq+1}\to a_1$ as $n\to\infty$ and
define $\varphi_m(\lambda):= \lim_{n\to\infty} h_{\lambda,m}(c_{nq+1})-a_1:= h_{\lambda,m} (a_1)-a_1$. Then
\begin{equation}\label{eqn:varphim}
G_{c_1+\lambda}^q(\varphi_m(\lambda)+a_1)=\varphi_m(\lambda) +a_1+O(\lambda^{m+1}).
\end{equation}


\begin{lemma}\label{lem:61}  There is a function $\varphi(\lambda)$, holomorphic near $\lambda=0$ and $m_0$, such that
$$G_{c_1+\lambda}^q (\varphi(\lambda)+a_1)=\varphi(\lambda)+a_1,$$
and such that for each $m\ge m_0$,
\begin{equation}\label{eqn:varphimvarphi}
\varphi_m(\lambda)-\varphi(\lambda)=O(|\lambda|^{m/3}) \mbox{ as } \lambda\to 0.
\end{equation}
\end{lemma}
\begin{proof}
In the case $\kappa=Dg^q(a_1)\not=1$, we simply take $\varphi(\lambda)$ so that $\varphi(\lambda)+a_1$ is the fixed point of $G_{c_1+\lambda}^q$ obtained as analytic continuation of $a_1$. It is easy to check that (\ref{eqn:varphimvarphi}) holds for all $m\ge 1$ with an even better error term: $O(\lambda^{m+1})$ instead of $O(|\lambda|^{m/3})$.

Now assume $\kappa=1$ and consider the map $\Phi(\lambda, z)= G_{c_1+\lambda}^q (z+a_1)-z-a_1$ which is holomorphic in a neighborhood of the origin in $\C^2$. Clearly, $\Phi(0,z)$ is not identically zero, so by Weierstrass' Preparation Theorem, 
there is a Weierstrass polynomial $$Q(\lambda,z)=z^2+ 2u(\lambda) z+ v(\lambda)=(z+u(\lambda))^2+v(\lambda)-u(\lambda)^2$$
such that $\Phi(\lambda, z)=Q(\lambda, z) R(\lambda, z)$, where $R$ is a holomorphic function near the origin and $R(0,0)\not=0$, and $u, v$ are holomorphic near the origin in $\C$ with $u(0)=v(0)=0$. Consider the discriminant $\Delta(\lambda)=u(\lambda)^2-v(\lambda)$ which satisfies $\Delta(0)=0$.
By (\ref{eqn:varphim}), for each $m$, 
\begin{equation}\label{eqn:varphim1}
(\varphi_m(\lambda)+u(\lambda))^2 -\Delta (\lambda)=O(\lambda^{m+1}).
\end{equation}

Let us distinguish two cases.

{\em Case 1.} $\Delta(\lambda)\equiv 0$ (for $\lambda\in \D_\eps$ for some $\eps>0$). Then $-u(\lambda)$ is the only zero of $Q(\lambda, z)$ near $0$. By (\ref{eqn:varphim1}), $\varphi_m(\lambda)+u(\lambda)=O(|\lambda|^{(m+1)/2})$. So the claim holds with $\varphi=-u$.

{\em Case 2.} $\Delta(\lambda)\not\equiv 0.$ Then there is $n_0\ge 1$ and $A\not=0$ such that $\Delta(\lambda)= A\lambda^{n_0}+O(\lambda^{n_0+1})$.
Assume $m\ge n_0$. 
By (\ref{eqn:varphim1}), 
$$(\varphi_m(\lambda)+u(\lambda))^2 =A\lambda^{n_0} +O(\lambda^{n_0+1}),$$
which implies that $n_0$ is even.
Therefore, there exists holomorphic functions $\varphi_{\pm}(\lambda)$ 
such that
$Q(\lambda, z)=(z-\varphi_+(\lambda))(z-\varphi_-(\lambda))$.
The lemma holds for either $\varphi(\lambda)=\varphi_-(\lambda)$ or $\varphi(\lambda)=\varphi_-(\lambda)$.
\end{proof}

\begin{proof}[Completion of the proof of Theorem C]

We want to show that $\sigma(\lambda):=DG_{c_1+\lambda}^ q( \varphi(\lambda)+a_1)$ is constant. Arguing by contradiction, assume that this is not the case. Then there exists $m_1\ge 1$ and $A\not=0$, such that
$$\sigma(\lambda)-\sigma(0)=3A\lambda^{m_1}+O(\lambda^{m_1+1}).$$
Fix $m\ge \max(3(m_1+1),m_0)$. There exists $\eps'=\eps_m'>0$ such that the following hold:
\begin{itemize}
\item There exist admissible holomorphic motions $h_{\lambda,m}^{k}$ of $\overline{P}_{r_0}$ over $\D_{2\eps'}$, $k=0,1,\ldots,$ such that $h_{\lambda,m}^{0}=h_{\lambda,m}$, $h_{\lambda,m}(c_1)=c_1+\lambda$ and such that $h_{\lambda,m}^{k+1}$ is the lift of $h_{\lambda,m}^k$ for each $k$, and the
 sequence $\{h_{\lambda,m}^{k}(x)\}$ is uniformly bounded in $(\lambda,x)\in
    \D_{2\eps'}\times \overline{P}_{r_0}$.
\end{itemize}
Thus, there exists $C=C_m>0$ such that whenever $|\lambda|<\eps'$ and $n\ge 0$, 
\begin{equation}\label{eq:asymp1}
|h_{\lambda,m}^{n} (c_1)-c_1-\lambda|\le C |\lambda|^{m+1},
\end{equation}
\begin{equation}\label{eq:asymp2}
|h_{\lambda,m}^{n} (a_1)-\varphi_m(\lambda)-a_1|\le C|\lambda|^{m+1}.
\end{equation}
By Lemma~\ref{lem:61}, enlarging $C$ if necessary, we have
\begin{equation}\label{eq:asymp3}
|\varphi(\lambda)-\varphi_m(\lambda)|\le C|\lambda|^{m/3}\le C|\lambda|^{m_1+1}.
\end{equation}
Put $$\mathcal{G}_\lambda^{(n)}(z)=G_{h_{\lambda,m}^{n}(c_1)}\circ G_{h_{\lambda,m}^{n+1}(c_1)}\circ \cdots \circ G_{h_{\lambda,m}^{n+q-1}(c_1)}(z).$$
Then $h_{\lambda,m}^{(n)}(g^q(z))=\mathcal{G}_\lambda^{(n)}(h_{\lambda,m}^{(n+q)}(z))$
and the sequence $\{\mathcal{G}_\lambda^{(n)}(z)\}$ is uniformly bounded in $\lambda\in\D_{\eps'}$, $z\in B(a_1,\delta_0)$ for some $\delta_0>0$.
By the lifting property (\ref{eq:asymp1}),(\ref{eq:asymp2}), (\ref{eq:asymp3}) 
enlarging $C$ further, we have
$$|D(\mathcal{G}_{\lambda}^{(n)})(h_{\lambda,m}^{n+q}(a_1))-\sigma(\lambda)|\le C |\lambda|^{m_1+1}.$$
for each $n\ge 0$.
It follows that there exists $\eps''>0$ such that whenever $|\lambda|\le \eps''$, we have
$$|D\mathcal{G}_{\lambda}^{(n)}(h_{\lambda,m}^{n+q}(a_1))-(\sigma(0)+3A\lambda^{m_1})|=O(|\lambda|^{m_1+1}).$$
Therefore, we can choose $\lambda\not=0$ with $|\lambda|$ arbitrarily small and such that
$$|D\mathcal{G}_{\lambda}^{(n)}(h_{\lambda,m}^{n+q}(a_1))|<|\sigma(0)|-2|A||\lambda|^{m_1}<|\sigma(0)|=|\kappa|$$
holds for every $n$. We fix such a choice of $\lambda$ now.

Let $\delta>0$ be a small constant such that for each $z$ and each $n\ge 0$ with $|z-h_{\lambda,m}^{n+q}(a_1)|<\delta$,
$$|D\mathcal{G}_\lambda^{(n)}(z)|<|\sigma(0)|-|A||\lambda|^{m_1}=: \kappa'<|\sigma(0)|=|\kappa|.$$
Let $l_0>0$ be large enough such that for each positive integer $l\ge l_0$ and any $n\ge 0$, $|h_{\lambda,m}^{n} (a_1)- h_{\lambda,m}^{n} (c_{lq+1})|<\delta$.
Then for any $l\ge l_0$, and any $n\ge 0$, 
\begin{multline*}
|h_{\lambda,m}^{n}(a_1)- h_{\lambda,m}^{n} (c_{(l+1)q+1})|
= |\mathcal{G}_{\lambda}^{(n)}(h_{\lambda,m}^{n+q} (a_1))-\mathcal{G}_{\lambda}^{(n)} (h_{\lambda,m}^{n+q}(c_{lq+1})|\\
\le  \kappa' |h_{\lambda,m}^{n+q} (a_1)-h_{\lambda,m}^{nq} (c_{lq+1})|.
\end{multline*}
It follows that
\begin{equation}\label{eqn:hlambdama1}
|h_{\lambda,m}^{0} (a_1)- h_{\lambda,m}^{0} (c_{lq+1})|=O({\kappa'}^l).
\end{equation}

Let us now distinguish two cases to complete the proof by deducing a contradiction.

{\em Case 1.} $|\kappa|<1$. In this case, $P_{r_0}$ contains a neighborhood of $a_1$, so by the admissible property of $h_{\lambda, m}$, we have $|h_{\lambda,m}^0(a_1)-h_{\lambda,m}^{0} (c_{lq+1}|\asymp |a_1-c_{lq+1}|\asymp |\kappa|^l$.  This leads to $|\kappa|<|\kappa'|$, a contradiction!

{\em Case 2.} $|\kappa|=1$. In this case, $c_{lq+1}$ converges to $a_1$ only polynomially fast. However, since $h_{\lambda,m}$ is qc and hence bi-H\"older, (\ref{eqn:hlambdama1}) implies that $c_{lq+1}$ converges to $a_1$ exponentially fast, a contradiction!
\end{proof}

\section{Applications to transversality for complex families} \label{sec:complex}

In this section we will show that
Theorems~\ref{thm:Main2} and \ref{thm:Main3} follow from the Main Theorem.
First we show that the attracting periodic cycle contains the singular value in its immediate basin of attraction.

Let $\R_+=(0,\infty)$.
\begin{prop} \label{prop:c1toa0}
Consider one of the following three situations:
\begin{enumerate}
\item [(i)] $g(z)=f(z)+c_1$ with $f\in\mathcal{F}$, $c_1\in\overline{D}$, and $\mathcal{O}$ is an attracting or parabolic cycle of $g$ with multiplier $\kappa\not=0$;
\item [(ii)] $g(z)= c_1 f(z)$ with $f\in\mathcal{E}$, $c_1\in D^+\setminus \{0\}$ and $\mathcal{O}\subset D\setminus \{0\}$ is an attracting or parabolic cycle of $g$ with multiplier $\kappa\not=0$;
\item [(iii)] $g(z)= c_1 f(z)$ with $f\in\mathcal{E}_o$, $c_1\in D^+\setminus \{0\}$ and $\mathcal{O}\subset D\setminus \{0\}$ is an attracting or parabolic cycle of $g$ with multiplier $\kappa\not=0$ but $\mathcal{O}\not=-\mathcal{O}$;
\end{enumerate}
Then
\begin{enumerate}
\item there is a simply connected open set $W$ with the following property:
\begin{itemize}
\item $g^{kq}$ is univalent on $W$ for each $k\ge 1$,
\item $g^{kq}$ converges uniformly on $W$ to a point in $\mathcal{O}$;
\item $W\ni c_1$ in cases (i) and (ii), while $W\ni c_i$ or $W\ni -c_i$ in case (iii).
\end{itemize}
\item if $\mathcal O$ is a parabolic periodic orbit then it is non-degenerate.
\end{enumerate}
\end{prop}

\newcommand{\E}{E}
\begin{proof}
We shall prove this proposition by the classical Fatou argument. Let $\E=\{c_1\}$ in case (i) and (ii) and let $\E=\{\pm c_1\}$ in case (iii).


Note that the assumption implies that $0\not\in \mathcal{O}$ in all cases.
Take a Jordan disk $A_0\subset D\setminus \{0\}$ along with a univalent map $\varphi:A_0\to \C$ (Koenigs or Fatou coordinate) so that $g^{q'}(A_0)\subset A_0$ where (i) in attracting case: $a_0\in A_0$, $q'=q$, $\varphi(A_0)=\D_\rho$ for some $\rho>0$ and $\varphi\circ g^q=\kappa \varphi$ on $A_0$, (ii) in the
parabolic case: $a_0\in \partial A_0$, $q'=rq$ for a minimal $r\ge 1$, $\varphi(A_0)=\{z: \Re(z)>M\}$ for some $M>0$ and
$\varphi\circ g^{q'}=\varphi+1$ on $A_0$.

In each case, we observe that for any connected open set $B\subset D\setminus (E\cup\{0\})$, $g: g^{-1}(B)\to B$ is an un-branched covering and $g^{-1}(B)\subset D\setminus \{0\}$.
Here we use that in case (ii), (iii) that $c_1\in D_+$ since this implies by  assumption  (c) in the definition of $\mathcal{E},\mathcal{E}_o$,
that $D/{c_1}\subset V$.

Let $A_n$, $n=1,2,\ldots$, denote the component of $g^{-n}(A_0)$ with $A_{q'k}\supset A_0$ for each $k=1,2,\ldots$ and $g(A_n)\subset A_{n-1}$.
There exists a minimal $N\ge 0$ such that 
$A_N\cap E\not=\emptyset$. Indeed, otherwise, since $A_0$ is a simply connected open set contained in $D\setminus (E\cup\{0\})$, we obtain from the observation above by induction that
$A_n\subset D\setminus (E\cup\{0\})$ and $g: A_n\to A_{n-1}$ is an un-branched covering for each $n\ge 1$. It follows that $g^{kq'}: A_{kq'} \to A_0$ is an un-branched covering, hence a conformal map for each $k=1,2,\ldots$. Thus $\varphi$ extends to a univalent function from $A=\bigcup_{k=0}^\infty A_{kq}$ onto $\C$ via the functional equation $\varphi(f^{q'}(z))=\kappa \varphi(z)$ or $\varphi(f^{q'}(z))=\varphi(z)+1$, which implies that $A=\C$, contradicting with $E\cap A=\emptyset$. Taking $W=A_N$ completing the proof of (1).

Let us now prove (2). So assume that $\mathcal{O}$ is parabolic. In case (i) and (ii), as the argument above shows that each attracting petal around $\mathcal{O}$ intersects the orbit of $c_1$, the cycle is non-degenerate. Assume now that we are in case (iii). Then either $c_1$ or $-c_1\in W$. Since the map $g$ is odd and $\mathcal{O}\not=-\mathcal{O}$, only one of $c_1$ and $-c_1$ is contained in the basin of attraction of $\mathcal{O}$, and thus the statements (2) hold for the same reason as before.
\end{proof}

\begin{proof}[Proof of Theorems~\ref{thm:Main2} and~\ref{thm:Main3}]
The proposition above implies that  $g$ (restricted to a suitable small domain $U$) is a marked map  w.r.t. $c_1$.
Choosing $r_0>0$ small enough, we have $P_{r_0}$ is compact subset of $U$.
Then, as is in shown in \cite{LSvS1,LSvS1a}
the lifting property~\ref{def:liftingproperty} holds. Indeed,  let $h^{(0)}_\lambda:=h_\lambda \colon P_{r_0}\to \C$, $\lambda\in \D_\epsilon$ be a holomorphic motion.
As is shown in \cite{LSvS1,LSvS1a}, one can define a sequence of holomorphic motions $h^{(n)}_\lambda$
as in (2) of Definition~\ref{def:liftingproperty} so that properties (1),(3) also hold. So $(g,G)_W$ has the
lifting property for the set $P_{r_0}$. In particular,  the first parts of Theorems~\ref{thm:Main2}-\ref{thm:Main3} follow
from the Main Theorem.
\end{proof}

\section{Application to families of real maps}\label{sec:realmaps}


In this section, we shall prove Theorem~\ref{thm:eremenko}. To this end, we shall need the result in~\cite{LSvS1,LSvS1a} 
to determine the sign of $\kappa'$ and $Q$ in the transversality inequalities in Theorems~\ref{thm:Main2} and~\ref{thm:Main3}.

Throughout let $f_t$ be a family as in the assumption of Theorem~\ref{thm:eremenko}, case (i) or (ii). The case (iii) can be easily reduced to case (i). Put $c=0$ in case (i), so that $c$ is the common turning point of $f_t$. For $t< c$, $f_t$ has no periodic point of period greater than one, so in the following, we shall mainly concerned with
$$t\in J_+:=(c,\infty)\cap J.$$ The following is an immediate consequence of Proposition~\ref{prop:c1toa0}.
\begin{prop}\label{prop:realc12a0}
Suppose that $f_{c_1}$, $c_1\in J^+$, has an attracting or parabolic cycle $\mathcal{O}\subset J$ and let $a_0$ be the rightmost point in $\mathcal{O}$. Then $f^{kq}$ is monotone on $[a_0,c_1]$ for all $k\ge 0$ and $f^kq$ converges uniformly on $[a_0, c_1]$ to $a_0$ as $k\to\infty$.
\end{prop}

\begin{proof} Let $\tilde{f}$ denote the complex extension in $\mathcal{F}$ or $\mathcal{E}\cup\mathcal{E}_o$. Since $\tilde{f}$ is real symmetric, the simply connected domain $W$ as claimed in Proposition~\ref{prop:c1toa0} can be taken to be symmetric with respect to $\mathbb{R}$. If $c_1\in W$, then the statement follows. If $c_1\not\in W$, then $f\in\mathcal{E}_{o,u}$, $-c_1\in W$ and there exists $a_0\in \mathcal{O}$ such that $f^{kq}$ converges to $a_0$ on the interval $K$ bounded by $a_0$ and $-c_1$. However, $a_0>0$ and $-c_1<0$ so $K\ni 0$. Since $f(0)=0$, this is absurd!
\end{proof}

\begin{prop}\label{prop:kappa=0}
If $g:=f_{c_1}$ has a cycle $\mathcal{O}$ with multiplier $0$, and let $\kappa(t)$ denote the multiplier of the attracting cycle of $f_t$ for $t$ close to $c_1$. Then there exists $\eps>0$ such that $\kappa(t)> 0$ for $t\in (c_1-\eps, c_1)$ and $\kappa(t)<0$ for $t\in (c_1, c_1+\eps)$.
\end{prop}
\begin{proof} In~\cite{LSvS1}, \cite{LSvS1a} the following inequality was proved
\begin{equation}\frac{ \left.\frac{d}{d t} f_t^q(c)\right|_{t=c_1}}{Dg^{q-1}(c_1)}>0 .
\label{eq:postrans22}
\end{equation}
Let us deduce the conclusion of the proposition. Let $a(t)$ denote the fixed point of $f_t^q$ near $c$ for $t$ close to $c_1$. Assume for definiteness that $Dg^{q-1}(c_1)>0$. Then it follows that there is $\eps>0$ such that $f_t^q(c)>c$ for $t\in (c_1,c_1+\eps)$ and
$f^q_t(c)<c$ for $t\in (c_1-\eps, c_1)$.
Thus for $t\in (c_1,c_1+\eps)$, $a(t)>c$ and for $t\in (c_1-\eps, c_1)$, $a(t)<c$. Reducing $\eps$ if necessary, $Dg_t^{q-1}(g_t(a(t))>0$, so $\kappa(t)= Dg(a(t)) Dg^{q-1}(g_t(a(t))$ has the sign as claimed.
\end{proof}

We say an open subinterval $J_1$ of $J$ is an {\em attracting window} if for each $t\in J_1$, $f_t$ has an attracting periodic cycle $\mathcal{O}_t$ with multiplier $\kappa(t)\in (-1,1)$.
By the Implicit Function Theorem, $\mathcal{O}_t$ and $\kappa(t)$ depending on $t$ in a $C^1$ way.
In particular, the cycles $\mathcal{O}_t$ have the same period, which is called the period of the attracting window.

\begin{lemma}\label{lem:attrwin}
Let $J_1$ be an attracting window of period $q\ge 2$ and let $\kappa(t)$ be the corresponding multiplier function. Then $\kappa(t)$ is strictly decreasing in $J_1$.
\end{lemma}
\begin{proof} Assume without loss of generality that $J_1=(t_-, t_+)$ is a maximal attracting window.
First, by Lemma~\ref{prop:kappa=0}, $\kappa$ can have at most one zero in $J_1$. Indeed, otherwise, let $c_1<\hat{c}_1$ be two consecutive zeros, then it follows that $\kappa<0$ in a right neighborhood of $c_1$ and $\kappa>0$ in a left neighborhood of $\hat{c}_1$, which implies by the intermediate value theorem that there is another zero of $\kappa$ in-between $c_1$ and $\hat{c}_1$, absurd!

Now let us assume that $\kappa(t_0)=0$ for some $t_0\in J_1$. Then $\kappa(t)>0$ for $t\in (t_-, t_0)$ and $\kappa(t)<0$ for $t\in (t_0, t_+)$. By Theorems~~\ref{thm:Main2} and~\ref{thm:Main3}, $\kappa'(t)\not=0$ for each $t\not=t_0$. This forces $\kappa'(t)<0$ for all $t\in J_1\setminus \{0\}$, so $\kappa(t)$ is strictly decreasing in $J_1$.

Finally, let us prove that $\kappa$ does have a zero in $J_1$. Indeed, it is easy to see $t_-,t_+\in J$, so by the maximality of $t_+$, we have $\kappa(t)\to \pm 1$ as $t\nearrow t_{\pm}$. Since $\kappa$ is monotone in $J_1$ (by Theorems~~\ref{thm:Main2} and~\ref{thm:Main3}), the intermediate value theorem implies that $\kappa$ has a zero.
%
%
\end{proof}

\begin{prop}\label{prop:saddlenode}
Assume that for some $c_1\in J_+$, $g=f_{c_1}$ has a parabolic cycle with multiplier $1$. Let $q$ be the period of the cycle and let $a_0$ be the rightmost point in the cycle.
Then
there exist $\eps>0$ and $\delta>0$ such that for each $t\in (c_1, c_1+\eps)$, $f_t$ has an attracting cycle of period $q$ and $a_0(t)\to a_0$ as $t\to c_1$, and for each $t\in (c_1-\eps, c_1)$ and $x\in [a_0-\delta, a_0+\delta]$, $f_t^q(x)<x$. Equivalently, for $a_0:=\lim_{k\to\infty} g^{kq}(c_1)$, we have 
$$Q(a_0):=\dfrac{d}{dt} f_t^q(a_0)\bv_{t=c_1}>0.$$
\end{prop}

\begin{proof} By Proposition~\ref{prop:realc12a0}, $c_{kq+1}$ decreases to $a_0$ and $D^2 g^q(a_0)<0$.
By Theorems~\ref{thm:Main2} and~\ref{thm:Main3},
$Q(a_0)\not=0$. So $f_t^q(x)$ displays a saddle-node bifurcation at $(a_0, c_1)$. It is well-known, see for example~\cite[Proposition 7.7.5]{BS} that there exist $\eps>0$ and $\delta>0$ such that for each $t$ in
$$J_1=\left\{\begin{array}{ll}
[c_1, c_1+\eps) &\mbox{ if } Q(a_0)>0,\\
(c_1-\eps, c_1] &\mbox{ if } Q(a_0)<0,
\end{array}
\right.$$
$f_t^q$ has two fixed points $a_-(t)$ and $a_+(t)$, depending continuously in $t$, with $a_-(c_1)=a_+(c_1)=a_0$, and with $0<Df_t^q(a_-(t))<1$ and $Df_t^q(a_+(t))>1$, while for $t$ in
$$J_2=\left\{\begin{array}{ll}
(c_1-\eps, c_1) &\mbox{ if } Q(a_0)>0,\\
(c_1, c_1+\eps) &\mbox{ if } Q(a_0)<0,
\end{array}
\right.$$
$f_t^q(x)<x$ for each $|x-a_0|\le \delta$. 
So $J_1^o$ is an attracting window, and by Lemma~\ref{lem:attrwin}, the multiplier function $\kappa(t)=Df_t^q(a_-(t))$ is monotone decreasing.
Thus $J_1=[c_1, c_1+\eps)$ and $Q(a_0)>0$.
\end{proof}

\begin{prop}\label{prop:pd}
Assume that $f_{c_1}$ has a parabolic cycle $\mathcal{O}$ of period $q$ and with multiplier $-1$. Then for any positive integer $N\ge 1$, there exist $\eps>0$ and $\delta>0$ such that for each $t\in (c_1-\eps, c_1+\eps)$, $f_t^{2Nq}$ has exactly three fixed points in the $\delta$-neighborhood of $\mathcal{O}$, two of them hyperbolic attracting, and one hyperbolic repelling.
\end{prop}
\begin{proof} Note that the assumption implies that $c_1\in J_+$. For the case $N=1$, this is a well-known fact about periodic doubling bifurcation.  Indeed, by Theorems~\ref{thm:Main2} and ~\ref{thm:Main3}, $\kappa'(c_1)\not=0$, where $\kappa(t)$ is the multiplier of the periodic orbit of $f_t$ of period $q$ near $\mathcal{O}$ for $t$ close to $c_1$. We must have $\kappa'(c_1)<0$ for otherwise, a small right-sided neighborhood of $c_1$ would be an attracting window on which the multiplier function is increasing which is ruled out by Lemma~\ref{lem:attrwin}. On the other hand, by Proposition~\ref{prop:c1toa0}, $D^3 g^{2q}(a)\not=0$ for each $a\in \mathcal{O}$. Thus the statement follows, for example, by  ~\cite[Proposition 7.7.6]{BS}.

The case for general $N$ follows: $D^3 g^{2qN}(a)\not=0$ for $a\in\mathcal{O}$. For reducing $\eps,\delta$ if necessary, $f^{2Nq}$ has at most three fixed points in the $\delta$-neighborhood of $\mathcal{O}$ for $t\in (c_1-\eps, c_1)$.
\end{proof}

\begin{proof}[Proof of Theorem~\ref{thm:eremenko}] Let $\kappa$ be the multiplier the $f_{t_*}$-cycle $\mathcal{O}$.
First let us assume $\kappa\not=1$. Then by the Implicit Function Theorem, there is a maximal $T$ such that $J_1:=[c_1, T)\subset J$ and each $f_t$, $t\in J_1$, has a periodic cycle $\mathcal{O}_t$ which depends on $t$ continuously and such that $\mathcal{O}_t=\mathcal{O}$ and the multiplier $\kappa(t)\not=1$. We shall prove that $T$ is the right endpoint of $J$. Arguing by contradiction, assume that this is not the case. By the maximality of $T$, $\lim_{t\to T} \kappa(t)=1$.
By the intermediate value theorem, we are in one of the following cases:

{\em Case 1.} $\kappa(t)>1$ for all $t\in J_1$. Choose a subsequence $t_n\nearrow T$ such that $\mathcal{O}_{t_n}$ converges to a periodic orbit $\mathcal{O}_T$ of $f_T$. Let $q$ denote the period of $\mathcal{O}$ and $q'$ the period of $\mathcal{O}_T$, then $q=q' m$
for some positive integer $m$.

If $\mathcal{O}_T$ has multiplier $1$, then by Proposition~\ref{prop:saddlenode},  there exist $\eps>0$ and $\delta>0$ such that $f_{t}^{q'}(x)<x$ for each $t\in (T-\eps, T)$ and $|x-a_0(T)|<\delta$, where $a_0(T)$ denotes the rightmost point in the parabolic cycle $\mathcal{O}_T$. By continuity, there exist $\eps',\delta'>0$, we have $f_t^q(x)<x$ when $t\in (T-\eps',T)$ and $|x-a_0(T)|<\delta'$. This is in contradiction with the construction of $\mathcal{O}_T$.

Assume now that $\mathcal{O}_T$ has multiplier $-1$. Then $q=2q'N$ for some integer $N\ge 1$. By Proposition~\ref{prop:pd}, for each $n$ large enough, $f_{t_n}^{q}$ has exactly three fixed points, two of which are attracting. However $\mathcal{O}_{t_n}$ contains at least two repelling fixed point of $f_{t_n}^q$, a contradiction!

{\em Case 2.} $\kappa(t)<1$ for all $t\in J_1$. Then there exists $\eps>0$ such that $(T-\eps, T)$ is an attracting window and by Lemma~\ref{lem:attrwin}, $\kappa(t)$ is strictly decreasing, thus contradicting the maximality of $T$.

So now let us consider the case $\kappa=1$. By Proposition~\ref{prop:saddlenode}, there is $\eps>0$ such that $(c_1, c_1+\eps)$ is an attracting window. By what have proved above, the attracting periodic orbits allow further continuation until the right endpoint of $J$.
\end{proof}
\begin{remark} Figure~\ref{fig1} shows the bifurcation diagram  
for the family $f_{c_1}(x)=c_1 \sin(x)$, $c_1\in \R$. It also shows that for this family
there are degenerate parabolic bifurcations, which occur when $\mathcal{O}=-\mathcal{O}$.
\end{remark}

\begin{figure}[h!]
\centering
\includegraphics[width=0.9\textwidth,height=7cm]{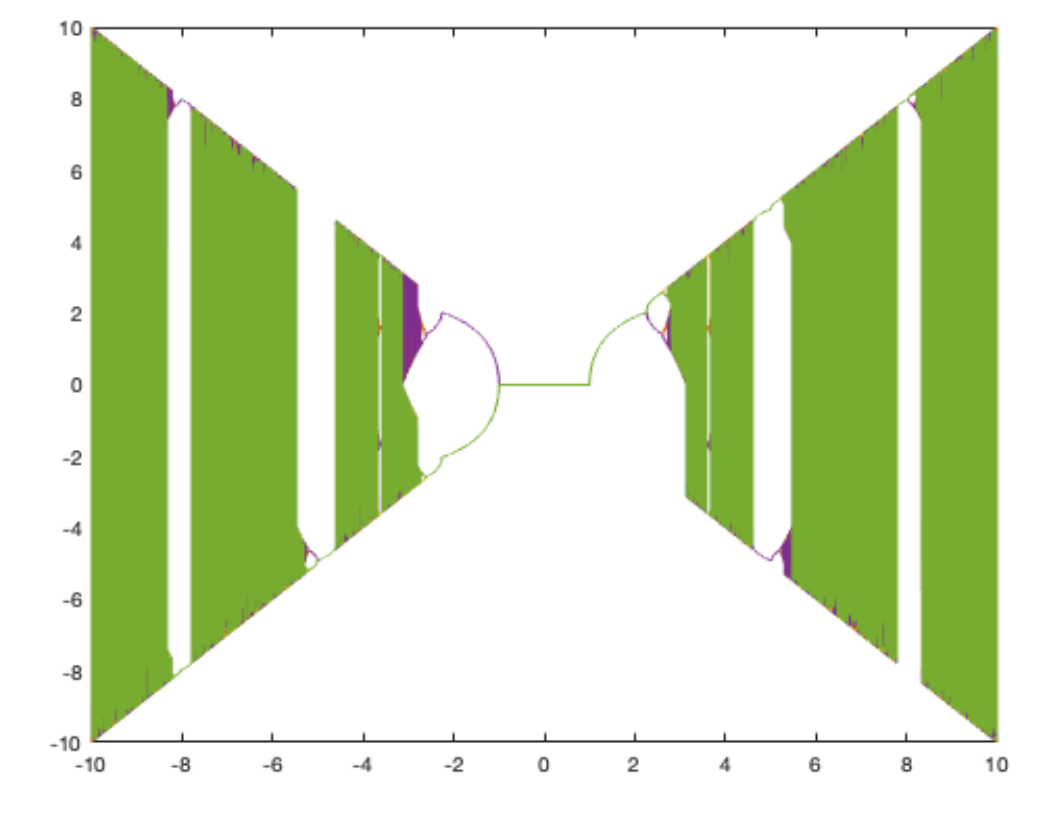}
\setlength{\abovecaptionskip}{-2pt}
\caption{\small  The bifurcation diagram of $f_c(x)=c\sin(x)$.  For each $c\in [-10,10]$,
the last 100 iterates from the set $\{f^k_c(\pi/2)\}_{k=0}^{1000}$ are drawn (in the vertical direction).
The bifurcation points for which only {\lq}one half of a parabola{\rq} is visible, correspond
to pitchfork bifurcations (where both critical values are attracted to a single parabolic periodic point), and where two halves are visible correspond
to period doubling bifurcations. Note that $\sin \in \mathcal{E}_o$, where $D=D_+=V=\C$.} \label{fig1}
\end{figure}

%
%

\appendix
\section{In the real parabolic case $\frac{\partial G^q_w}{\partial w} (a_0)\big\vert_{w=c}\ge 0$ holds}

We assume that $g$ is real marked map w.r.t. $c_1$ so that
the sequence $c_{n}:=g^{n-1}(c_1)$, $n=1,2,\dots$ tends to a non-degenerate parabolic periodic orbit $\mathcal O:=\{a_0,\dots,a_{q-1}\}$
with multiplier $+1$.
Consider a holomorphic deformation $(g,G)_W$ 
and assume that $(g,G)_W$ has the lifting property for the set $P=\{c_n\}_{n\ge 1}$
 (notice that this condition is local and weaker than the one of the Main Theorem).
Under these assumptions we will show that
 $Q(c_1)=\frac{\partial G^q_w}{\partial w}\big\vert_{w=c_1} (a_0)\ge 0$.  

This Appendix will also motivate the choice of the particular vector field along $P$ appearing in Section~\ref{sec:ThmA}.

From~(\ref{eq:partGL})  it follows that the 2nd equality holds in
 $$\Delta(z):= \sum_{j=1}^q \dfrac{L(g^{j-1}(z))}{Dg^j(z)}=\dfrac{1}{Dg^q(z)}\dfrac{\partial G^q_w}{\partial w}\bv_{w=c} (z)$$
Since $Dg^q(a_0)=1$,
$$\Delta(a_0)=\dfrac{\partial G^q_w}{\partial w}\bv_{w=c} (a_0) = \sum_{j=1}^q \dfrac{L(g^{j-1}(a_0))}{Dg^j(a_0)} = \sum_{j=0}^{q-1} \dfrac{L(g^j(a_0))}{Dg^{j+1}(a_0)}  .$$

\begin{prop}\label{prop:ge0}   Assume that $(g,G)_W$ has the lifting property for the set $P$. 
Moreover, assume that $g$ is real, has a periodic point  $a_0$ of period $q$ with $Dg^q(a_0)=1$, $D^2g^q(a_0)\neq 0$ so that
$c_1$ is in the basin of $a_0$ and so that $Dg^q(c_{kq+1})>0$ for each $k\ge 0$. 
Then
$$\dfrac{\partial G^q_w}{\partial w}\bv_{w=c} (a_0)\ge 0.$$
\end{prop}
\begin{proofof}{Proposition~\ref{prop:ge0}} By  Proposition~\ref{prop:32}, see also \cite{Le1},
\begin{equation}
D(\rho):=1+\sum_{n=1}^\infty \dfrac{\rho^nL(c_n)}{Dg^n(c_1)} >0\mbox{ for all } 0<\rho < 1.
\label{drho}
\end{equation}


Arguing by contradiction, assume that $\Delta(a_0)<0$. 
Then there exists $\rho_0\in (0,1)$ and $k_0\ge 1$ such that for each $\rho_0<\rho<1$ and each $k\ge k_0$, we have
\begin{equation}\label{eqn:Bkrho}
B_k(\rho):=\sum_{j=1}^q \dfrac{\rho^j L(c_{qk+j})}{Dg^{j}(c_{qk+1})}< \Delta(a_0)/2.
\end{equation}

By assumption, $M_k:=k^2 Dg^{kq}(c_1)>0$ holds for all $k\ge 0$. By Leau-Fatou Flower theorem, see Appendix B, $$Dg^q(c_{kq+1})=1-2/k +O(\log k \,\, k^{-2}),$$ 
so $M:=\lim_{k\to\infty} M_k>0$ exists. Enlarging $k_0$, we have $M_k> M/2$ for all $k\ge k_0$.
Then
\begin{align*}
D(\rho) &= 1+\mathlarger{\sum}_{n=1}^\infty \dfrac{\rho^nL(c_n)}{Dg^n(c)}  = 1+  \mathlarger{\sum}_{k=0}^\infty \mathlarger{\sum}_{j=1}^{q}\dfrac{\rho^{qk+j}L(c_{qk+j})}{Dg^{qk+j}(c)}=\\
  & \\
  &=1+\mathlarger{ \mathlarger{\sum}}_{k=0}^\infty \left[ \frac{\rho^{qk}}{Dg^{qk}(c_1)} \sum_{j=1}^{q}  \dfrac{\rho^j L(c_{qk+j})}{Dg^{j}(c_{qk+1})} \right] =1+\mathlarger{ \mathlarger{\sum}}_{k=0}^\infty \frac{k^2\rho^{qk}}{M_k} B_k(\rho)\\
  & \le 1+\mathlarger{\sum}_{k=0}^{k_0} \frac{k^2\rho^{qk}}{M_k} B_k(\rho) +\frac{\Delta(a_0)}{M} \sum_{k=k_0+1}^\infty k^2 \rho^{qk},
\end{align*}
provided that $\rho_0<\rho<1$.
Since $\sum_{k=k_0+1}^\infty k^2 \rho^{qk} \to \infty$ as $\rho\to 1$, this implies that $\liminf_{\rho \nearrow 1} D(\rho)=-\infty$,  a contradiction
with $D(\rho)>0$ for all $\rho \in (0,1)$.
 \end{proofof}

\subsection{A vector field $v_\rho$  along $P$ so that $\rho \mathcal A v_\rho = v_\rho$}
\begin{prop}\label{prop:32}
Let $g$ be a marked map w.r.t. $c_1$, and  $P=\{c_n\}_{n\ge 1}$ converges to the periodic orbit $\mathcal O$ of $g$
which has the multiplier $+1$ and not degenerate.
Consider a holomorphic deformation $(g,G)_W$.  
  Assume that $(g,G)_W$ has the lifting property for the set $P$.
  Then for all $|\rho|<1$ one has $$1+\sum_{n\ge 1}  \dfrac{\rho^nL(c_n)} {Dg^n(c_1)}\ne 0$$
 where $L(x)=\frac{\partial G_w}{\partial w}|_{w=c_1}(x)$.
\end{prop}
\begin{proof}
Let us for the moment assume that $h_\lambda(c_i)=c_i+v_i\lambda + O(\lambda^2)$ defines a holomorphic motion of $P$. 
Then its lift $\hat h_\lambda(c_i)=c_i+\hat v_i \lambda + O(\lambda^2)$ is defined for $|\lambda|$ small and
$$G_{h_\lambda(c_1))}(\hat h_\lambda(c_i))=h_\lambda(c_{i+1})=c_{i+1}+v_{i+1}\lambda+O(\lambda^2).$$
Writing $D_i=Dg(c_i)$
we obtain 
$$L_i v_1 + D_i\hat v_i = v_{i+1}, i\ge 1$$
where $L_i=L(c_i)$.
Taking $v=(v_1,v_2,\dots)$ and $\hat v=(\hat v_1,\hat v_2,\dots)$  we have that $\hat v=\mathcal{A}v$ where
$$\mathcal{A}=\left( \begin{array}{cccccc}
 -L_1/D_1 & 1/D_1  & 0  & \dots &  \dots  & \dots  \\
 -L_2/D_2   & 0 & 1/D_2 & 0 & \dots  & \dots \\
-L_3/D_3 & 0 &  0 & 1/D_3 & \dots   & \dots  \\ \vdots & \vdots & \vdots & \vdots & \ddots & \vdots   
\end{array}\right).$$

Assume $\rho\ne 0$ and consider the vector field $v_\rho$ on $P$ defined for $n>1$ by
\begin{equation*} \begin{array}{rl}
v_\rho(c_n) &:=\dfrac{Dg^{n-1}(c_1)}{\rho^{n-1}} \sum_{k=0}^{n-1} \dfrac{\rho^kL_k}{Dg^k(c_1)}=\\
& \\
&=L_{n-1}+\dfrac{Dg(c_{n-1})L_{n-2}}{\rho} + \dfrac{Dg^2(c_{n-2})L_{n-3}}{\rho^2 }+ \dots + \dfrac{Dg^{n-1}(c_1)L_0}{\rho^{n-1}}
\end{array}\end{equation*}
(with $L_0=1$)
and $v_\rho(c_1)=1$.
Notice that for this vector field we get $\rho \mathcal A v_\rho = v_\rho$.

Assume,  by contradiction,  that for some $0<|\rho|<1$,
$$1+\sum_{n\ge 1}  \dfrac{\rho^nL(c_n)} {Dg^n(c_1)}= 0.$$
Then we have that
$$v_\rho(c_n)= - \dfrac{Dg^{n-1}(c)}{\rho^{n-1}} \sum_{k=n}^{\infty} \dfrac{\rho^kL(c_k)}{Dg^k(c)}=  - \sum_{j=1}^\infty \dfrac{\rho^jL(c_{n+j-1})}{Dg^j(c_n)}. $$
For simplicity write $v_{i,\rho}=v_\rho(c_i)$. In the next lemma we will show that $v_\rho$ defines a Lipschitz vector field.
Because of this,
$v_\rho$ defines a holomorphic motion $h_{\lambda,\rho}$ for $|\lambda|<\epsilon$
for some $\epsilon>0$.  As $\rho$ is fixed, let us write $h_\lambda=h_{\lambda,\rho}$.
Since $(g,G)_W$ has the lifting property, it follows that the consecutive sequence of lifts $h^{(n)}_\lambda$
of $h_\lambda$ form a normal family. Write $h_\lambda^{(n)}(x)=x+\lambda v^{(n)}_\rho(x)  + O(\lambda^2)$. Then $v^{(n)}_\rho=\dfrac{1}{\rho^n}v_\rho$
which, since $\rho<1$, contradicts that the family $h_\lambda^{(n)}$ forms a normal family.
\end{proof}

 \begin{lemma}\label{liprho}
 The vector field
 $$v_\rho(c_n):= \sum_{j=1}^\infty \dfrac{\rho^jL(c_{n+j-1})}{Dg^j(c_n)}$$ defined on the set $P=\{c_i\}_{i\ge 1}$
 is Lipschitz.
 \end{lemma}
 \begin{proof}
 Given $x\in U$ such that $g^i(x)\in U$ for all $i\ge 1$,
 define $$V_\rho(x)=\sum_{j=1}^\infty \dfrac{\rho^jL(g^{j-1}(x))}{Dg^j(x)}.$$
 We have $V_\rho(c_n)=v_\rho(c_n)$, $n=1,2,\cdots$.
%
Moreover,
\begin{equation}\label{eq:v'}
V_\rho'(x)= \sum_{j=1}^\infty \rho^j\left[ \dfrac{L'(g^{j-1}(x))}{Dg(g^{j-1}(x)}-\dfrac{L(g^{j-1}(x))}{Dg^j(x)}  \sum_{i=0}^{n-1} \dfrac{D^2g(g^i(x))}{Dg(g^i(x))}\right].
\end{equation}
Since the periodic orbit $\mathcal O=\{a_0,\cdots,a_{q-1}\}$ of $g: U\to \C$
has multiplier $+1$ and is not degenerate, by the proof of Lemma~\ref{lem:leau-fatou}  for each $a_j$ there is a convex set $\Delta_j$ in the basin of $\mathcal O$ with a boundary point $a_j$ such that $g^q(\Delta_j)\subset \Delta_j$. Moreover, the closures of $\Delta_j$, $0\le j\le q-1$ are pairwise disjoint, all but finitely many points of the orbit $P$ are in $\Delta:=\cup_{j=0}^{q-1} \Delta_j$.
Since $g^{nq}$ converges uniformly on $\Delta_j$ to $a_j$, it follows that $Dg^j(x)\le C\rho^{-j/2}$, where $C$ is a constant.
These bounds along with the definition for $V_\rho$, (\ref{eq:v'}) and since $|\rho|<1$ imply that
for some $K>0$ and all $x\in \Delta$,
$$|V_\rho(x)|\le K, \ \ |V_\rho'(x)|\le K.$$
As each component $\Delta_j$ of $\Delta$ is convex (so that $x_1,x_2\in \Delta_j$ implies $|V_\rho(x_1)-V_\rho(x_2)|\le K|x_1-x_2|$) and only finitely many points of $P$ is outside of $\Delta$ we conclude that $V_\rho$ is Lipschitz on $P$.
\end{proof}

\section{The Leau-Fatou Flower}
Suppose that  $\mathcal{O}$ is a non-degenerate parabolic periodic orbit as above.
Fix $\alpha\in (0,1)$.
For each $j$, and $r>0$, define
$$\Theta_j^{\text{att}}=\{\theta\in [0,1): D^{\m +1} g^{\m q}(a_j) e^{2\pi i p\theta}\text{ is real and negative}\},$$
$$\Theta_j^{\text{rep}}=\{\theta\in [0,1): D^{\m+1} g^{\m q}(a_j) e^{2\pi i p\theta}\text{ is real and positive}\},$$
$$\cC_j(r)=\{a_j+se^{2\pi it}: 0<s<r, |t-\theta|<s^\alpha\text{ for some } \theta\in \Theta_j^{\text{rep}}\},$$
and
$$\cC'_j(r)=\{a_j+se^{2\pi it}: 0<s<r, |t-\theta|<s^\alpha\text{ for some } \theta\in \Theta_j^{\text{attr}}\},$$

The following is a variation of the well-known Leau-Fatou Flower Theorem.
\begin{lemma}[Leau-Fatou Flower Theorem]\label{lem:leau-fatou} \label{lem:Omega}
\begin{enumerate}
\item For each $r>0$, there exists $\tau=\tau(r)>0$ such that
$$B(a_j,\tau)\setminus \Omega_{r}\subset \mathcal{C}_j(\tau).$$
\item For any $r>\tau>0$ and $z_0\in \Omega_r^o$, there exists $n_0=n_0(z_0)$ such that 
$$g^n(z_0)\in \bigcup_j \cC'_j(\tau)$$ for all $n\ge n_0$.
\end{enumerate}
\end{lemma}
\begin{proof} This result is essentially contained in \cite{CG} or \cite{Milnor},
so we will contend ourself with a sketch in the case that $\mathcal{O}=\{0\}$ and $g(z)=z-z^{p+1}+O(z^{p+2})$. So $\Theta^{attr}=\{2\pi j/p: 0\le j<p\}$ and $\Theta^{rep}=\{\pi (2j+1)/p: 0\le j<p\}$.

Let us first prove (2). As described in \cite{Milnor} there are $p$ {\em attracting petals} $U_j$, $0\le j<p$, such that $U_j$ lies in the sector $(2j-1)\pi/\m<\theta< (2j+1)\pi/\m$ and such that for each $z_0\in U$ with $z_n:=g^n(z_0)\to 0$, $z_n\not=0$, there exists $n_0$ such that $z_n\in U_j$ for some $j$ and all $n\ge 0$. Let us prove that $z_n$ eventually lands in $\mathcal{C}'_j$. Indeed, assuming $j=0$ without loss of generality, and putting $w_n=-z_n^{-p}$, we have
$$w_{n+1}=w_n+1+O(|w_n|^{-1/p}).$$ From this, it is easy to check that for any $\tau>0$, $z_n\in \cC'(\tau)$ for all $n\ge n(\tau)$. The statement (2) is proved.

Let us prove the statement (1). First, by \cite{Milnor}, there exists $r_*>0$ such that if $\{w_{-n}\}_{n=0}^\infty$ is a $g$-backward orbit inside $\overline{B(0, r_*)}$, then $w_{-n}\to 0$. We may assume without loss of generality that $r\in (0, r_*)$. Next, we check that there exists $\tau_0>0$ such that for any $\tau\in (0,\tau_0)$ the set $\cC(\tau)$ is backward invariant under $g$: $g|\cC(\tau)$ is injective and $g(\cC(\tau)\supset \cC(\tau)$. Arguing by contradiction, assume that the statement (1) is false for some $r\in (0,r_*)$. Then for any $n\ge 1$, there is $z_{n}\in B(0,1/n)\setminus (\cC(1/n)\cup\{0\})$  and a minimal positive integer $k_n$ such that $g^{k_n}(z_n)\not\in B(0, r)$. Passing to a subsequence we may assume $g^{k_n-j}(z_n)\to w_{-j}$ as $n\to\infty$ for each $j$. Thus we obtain a $g$-backward orbit $\{w_{-j}\}_{j=1}^\infty$ with $w_{-j}\not\in \cC(\tau)$ and $w_{-j}\in \overline{B(0,r)}\setminus \cC(\tau)$ for $j\ge 1$. However, applying the statement (2) to $g^{-1}$, we see that this is impossible.
\end{proof}

\bibliographystyle{plain}             

\end{document}

\bibliographystyle{plain}             